\documentclass[11pt]{article}

\usepackage[english]{babel}
\usepackage{bm}
\usepackage{amsmath}
\usepackage{amsfonts, mathrsfs}
\usepackage{amssymb}

\usepackage{enumerate}

\usepackage[title]{appendix}

\usepackage{tikz}
\usepackage{graphicx}
\usepackage{epstopdf}
\usepackage{subcaption}

\numberwithin{equation}{section}
\newtheorem{thm}{Theorem}[section]
\newtheorem{cor}[thm]{Corollary}
\newtheorem{prop}[thm]{Proposition}
\newtheorem{lemma}[thm]{Lemma}

\newtheorem{assum}[thm]{Assumption}
\newtheorem{defn}[thm]{Definition}
\newtheorem{remark}[thm]{Remark}
\newenvironment{proof}[1][Proof]{\textbf{#1.} }{\ \rule{0.5em}{0.5em}}

\newcommand\numberthis{\addtocounter{equation}{1}\tag{\theequation}}

\def\Ind{\mathop{\hskip0pt{1}}\nolimits}

\newcommand{\ceil}[1]{\left\lceil#1\right\rceil}
\newcommand{\floor}[1]{\left\lfloor#1\right\rfloor}
\newcommand{\Z}{\mathbb{Z}}

\newcommand{\pr}{P}
\newcommand{\Prob}[1]{\pr\left(#1\right)}
\newcommand{\Probn}[1]{\mathbb{P}_n\left(#1\right)}
\newcommand{\CProb}[2]{\pr\left(\left.#1\right|#2\right)}
\newcommand{\CProbn}[2]{\mathbb{P}_n\left(\left.#1\right|#2\right)}

\newcommand{\e}{ E}
\newcommand{\Exp}[1]{\e\left[#1\right]}
\newcommand{\Expn}[1]{\mathbb{E}_n\left[#1\right]}
\newcommand{\CExp}[2]{\e\left[\left.#1\right|#2\right]}

\newcommand{\Zout}{Z^{+}}
\newcommand{\Zin}{Z^{-}}
\newcommand{\tbtZout}{\hat{Z}^{+}}
\newcommand{\tbtZin}{\hat{Z}^{-}}

\newcommand{\limWout}{W^-}
\newcommand{\limWin}{W^+}

\allowdisplaybreaks

\title{Typical distances in the directed configuration model}
\author{
	Pim van der Hoorn \\ {\small Northeastern University}
	\and
	Mariana Olvera-Cravioto \\ {\small University of California, Berkeley}
}

\begin{document}

\maketitle

\begin{abstract}
	We analyze the distribution of the distance between two nodes, sampled uniformly at random, in 
	digraphs generated via the directed configuration model, in the supercritical regime. Under the assumption that the covariance between the  in-degree and out-degree is finite, we show that the distance grows logarithmically in the size of the graph. In contrast with the undirected case, this can happen even when the variance of the degrees is infinite.  The main tool in the analysis is a new coupling between a breadth-first graph exploration process and a suitable branching process based on the Kantorovich-Rubinstein metric. This coupling holds uniformly for a much larger number of steps in the exploration process than existing ones, and is therefore of independent interest. 

\vspace{5mm}
\end{abstract}


\section{Introduction}

When proposing a mathematical model for studying the typical characteristics of complex networks, one of the first things to try to mimic is the degree distribution, i.e., the proportion of nodes having a certain number of neighbors. Perhaps the easiest way to do this, is by sampling a random graph from a prescribed degree sequence through the configuration or pairing model, originally introduced and analyzed in \cite{Bollobas1980, Wormald_78}. In the undirected case, the construction of the graph begins by assigning to each node a number of stubs or half-edges according to the given degree sequence, and determines the edges by randomly pairing the stubs,  each time by choosing uniformly among all the unpaired stubs. Conditionally on the resulting graph having no multiple edges or self-loops, it is well known that it has the distribution of a uniformly chosen graph among all those having the corresponding degree sequence (see, e.g., \cite{Bollobas_2001, VanDerHofstad2007}). In the directed setting, each node is given a number of inbound and outbound stubs according to its in-degree and out-degree, and the pairing is done by matching an inbound half-edge with an outbound one. Again, conditionally on having no self-loops or multiple edges in the same direction, the resulting graph is uniformly chosen among those having the prescribed degrees. 

The versatility of the configuration model and its ability to match any prescribed degree distribution makes it useful for analyzing the structural properties of networks as well as of processes on them \cite{Goh2003,Miller2009,Newman2002a, Chen2014}. One such property is the typical distance between nodes. In particular, for the undirected configuration model constructed from an i.i.d.~degree sequence, it is known that the hopcount between two randomly chosen nodes in a graph with $n$ nodes, conditioned on them being in the same component, grows logarithmically in $n$ when the degree distribution has finite variance \cite{Hofstad2005, Esker2008}, as $\log\log n$ when it has infinite variance but finite mean \cite{Hofstad2005a}, and is bounded if the mean is infinite \cite{Esker2005}. These results reflect 
what has been observed in many real networks, i.e., that the typical distance between connected nodes is very 
small compared to the size of the network, and that this distance gets shorter the more variable the degrees are.  

In this paper, we provide an analysis of the distance between two randomly chosen nodes in the supercritical directed configuration model\footnote{The supercritical regime ensures the existence of a giant strongly connected component.}, conditioned on the existence of a directed path from one to the other, under the assumption that the covariance between in- and out-degree is finite. We focus on the supercritical regime, since the existence of a directed path between two randomly selected nodes is a rare event in the critical and subcritical regimes. The directed nature of the graphs introduces some subtle differences compared to the undirected case, starting with the problem of constructing degree sequences having a prescribed joint distribution. More precisely, in the undirected configuration model one can obtain a degree sequence having distribution $F$ by simply sampling i.i.d.~observations from $F$ and adding one to the last node in case the sum is odd \cite{Arr_Ligg_05}. For the directed case, on the other hand, one needs to guarantee that the sum of the in-degrees is equal to that of the out-degrees, an event that can have asymptotically zero probability (e.g., when the in-degree and out-degree are allowed to be different, and nodes are independent). 

A more important difference between the undirected and directed cases is that the dependence between the in- and out-degree in the latter plays an important role in the behavior of the distance between nodes. More precisely, the main contribution of this paper is a theorem stating that the hopcount, i.e., the length of the shortest directed path between two nodes, grows logarithmically in the number of nodes, which unlike in the undirected case, can occur even when the variance of the degrees is infinite. Intuitively, the length of the shortest directed path between any two nodes will always be larger than the shortest undirected path. However, what is surprising, is that this distance does not necessarily get shorter as the variability of the degrees grows larger, and whether it gets shorter or not depends on the level of dependence between the in- and out-degree. Together with prior results on the existence and the size of a giant strongly connected component in random directed graphs \cite{Cooper2004, Penrose2014}, our results provide valuable insights into the differences and similarities between the directed and undirected cases. 

The second contribution of the paper is a novel coupling between a breadth-first graph exploration process and a Galton-Watson tree. This coupling is based on the Kantorovich-Rubinstein distance  between two probability measures (see, e.g., \cite{Villani2008}), and has the advantage of being uniformly accurate for a considerably longer time than existing constructions. Specifically, the coupling holds for a number of steps in the graph exploration process equivalent to discovering $n^{1-\epsilon}$ nodes, for arbitrarily small $\epsilon > 0$, compared to a constant number of nodes in \cite{Penrose2014}, $n^{1/2-\epsilon}$ nodes in \cite{Norros2006} and \cite{Durrett2010} (Theorem~2.2.2), or $n^{1/2 + \epsilon_0}$ nodes, for a very small $\epsilon_0 > 0$, in \cite{Hofstad2005, Esker2008}.  Moreover, the coupled branching process has a deterministic offspring distribution that does not depend on $n$ or the degree sequences, avoiding the need to consider intermediate tree constructions. The generality of our main coupling result, and the wide range of applications where a so-called branching process argument is used, makes it of independent interest.

The paper is organized as follows: Section \ref{S.NotationAndMainResult} contains an overview of our results for the typical distance between two randomly chosen nodes, with the main theorem presented in Section \ref{SS.MainResult}. The corresponding assumptions are given in terms of the realized degree sequences, i.e., the fixed degree sequences from which the graphs are constructed according to the pairing model. In Section \ref{S.Examples} we provide an algorithm that can be used to generate degree sequences satisfying our main assumptions for any prescribed joint distribution. We also include in that section numerical examples validating the accuracy of our theoretical approximations for the hopcount. 
Our coupling results are given in Section \ref{S.Coupling}, and in Section \ref{S.Distances} we give a more detailed derivation of the main theorem. All the proofs are postponed until Section \ref{S.Appendix}.


\section{Notation and main results} \label{S.NotationAndMainResult}

Throughout the paper we consider a directed random graph generated via the directed configuration model (DCM), that is, given two sequences $\{d_1^-, d_2^-, \dots, d_n^-\}$ and $\{d_1^+, d_2^+, \dots, d_n^+\}$ of nonnegative integers satisfying 
$$l_n = \sum_{i=1}^n d_i^- = \sum_{i=1}^n d_i^+,$$
we construct the graph by assigning to each node $i \in \{1, 2, \dots, n\}$ a number of inbound and outbound half-edges according to $(d_i^-, d_i^+)$, respectively. To determine the edges in the graph we pair each inbound stub with an outbound stub chosen uniformly at random among all unpaired stubs.  This pairing process is equivalent to matching the inbound half-edges with a permutation, uniformly chosen at random, of the outbound half-edges. We refer to the sequence $({\bf d}^-, {\bf d}^+) = (\{d_1^-, \dots, d_n^-\}, \{d_1^+, \dots, d_n^+\})$ as the bi-degree sequence of the graph. 

Our analysis of the typical distances in the DCM will be done in the large graph limiting regime, i.e., when the number of nodes $n \to \infty$. This means that we are in fact considering a sequence of graphs, indexed by $n$ each having its own bi-degree sequence, say $({\bf d}_n^-, {\bf d}_n^+) =  (\{d_{n,1}^-, \dots, d_{n,n}^-\}, \{ d_{n,1}^+, \dots, d_{n,n}^+\} )$.

As mentioned in the introduction, sampling a bi-degree sequence $({\bf d}_n^-, {\bf d}_n^+)$ having a prescribed joint distribution is not as simple as in the undirected case, so we allow the bi-degree sequence itself to be generated through a random process, as long as the realized bi-degree sequence satisfies our regularity conditions with high probability. To emphasize the possibility that the bi-degree sequence may itself be random, we will use the notation $({\bf D}^-_n, {\bf D}^+_n)$ to refer to the bi-degree sequence of a graph on $n$ nodes.  In particular, we use $D_i^-$ and $D_i^+$ to denote the in-degree and out-degree, respectively, of node $i$, and use $L_n = \sum_{i = 1}^n D_i^- = \sum_{i = 1}^n D_i^+$ to denote the total number of edges in the graph. To show that bi-degree sequences satisfying our main assumptions are easy to construct, we provide in Section~\ref{SS.IIDAlgorithm} an algorithm based on 
i.i.d.~samples from the prescribed degree distribution.

In view of our previous remarks, we need to be able to distinguish between the unconditional probability space and the conditional probability space given the bi-degree sequence $({\bf D}^-_n, {\bf D}^+_n)$. To this end, let $\mathscr{F}_n$ denote the sigma-algebra generated by the bi-degree sequence $({\bf D}^-_n, {\bf D}^+_n)$, and define $\mathbb{P}_n$ and $\mathbb{E}_n$ to be the corresponding conditional probability and expectation, respectively, given $\mathscr{F}_n$, i.e., $\mathbb{P}_n (\cdot) = E[ 1(\cdot) | \mathscr{F}_n]$ and $\mathbb{E}_n[ \cdot] = E[ \cdot | \mathscr{F}_n]$. We point out that under the probability $\mathbb{P}_n$, the bi-degree sequence is fixed, as in the classical configuration model.

Before we can state the assumptions imposed in our main theorems, we need to define the following (random) probability mass functions:
\begin{align*}
	g_n^+(t) &= \frac{1}{n} \sum_{r=1}^n 1(D_r^+ = t), \hspace{15mm} g_n^-(t) = \frac{1}{n} 
		\sum_{r=1}^n 1(D_r^- = t), \\
	f_n^+(t) &= \frac{1}{L_n} \sum_{r=1}^n 1(D_r^+ = t) D_r^-, \qquad f_n^-(t) = \frac{1}{L_n} 
		\sum_{r=1}^n 1(D_r^- = t) D_r^+,
\end{align*}
for $t = 0, 1, 2, \dots$, and let $G_n^+, G_n^-, F_n^+, F_n^-$ denote their corresponding cumulative distribution functions.  

We point out that the probability mass functions $g_n^+$ and $g_n^-$ correspond to the marginal distributions of the out-degree and in-degree, respectively, of a uniformly chosen node in the graph, while $f_n^+$ (resp. $f_n^-$) is the distribution of the out-degree (resp. in-degree) of a uniformly chosen inbound (resp. outbound) neighbor of that node, also known as the size-biased out-degree (resp. in-degree) distribution.  

{\bf Notation:} Throughout the manuscript we use the superscript $\pm$ to mean that the 
property/result holds for the distributions or random variables with the $\pm$ symbol substituted 
consistently with either the $+$ or $-$ symbol. 

The main assumption needed throughout the paper is given below. 

\begin{assum} \label{A.Wasserstein}
The  bi-degree sequence $({\bf D}^-_n, {\bf D}^+_n)$ satisfies:
\begin{enumerate}
	\item There exist probability mass functions $g^+, g^-, f^+$ and $f^-$ on the non negative 
	integers, such that, for some $\varepsilon > 0$, 
	\[
		\sum_{k=0}^\infty  \left| \sum_{i=0}^{k} \left( g_n^\pm(i)  - g^\pm(i) \right) \right| \leq 
		n^{-\varepsilon} \qquad \text{and} \qquad \sum_{k=0}^\infty  \left| \sum_{i=0}^{k} \left( 
		f_n^\pm(i)  - f^\pm(i) \right) \right| \leq n^{-\varepsilon},
	\]
	with $\nu \triangleq \displaystyle \sum_{j=0}^\infty j g^+(j) = \sum_{j=0}^\infty j g^-(j) < 
	\infty$ and $\mu \triangleq \displaystyle \sum_{j=0}^\infty j f^+(j) = \sum_{j=0}^\infty j 
	f^-(j) \in (1, \infty)$.
\item For some $0 < \kappa \leq 1$ and some constant $K_\kappa < \infty$, 
\[
	\sum_{r=1}^n ((D_r^-)^\kappa + (D_r^+)^\kappa) D_r^+ D_r^- \leq K_\kappa n.
\]
\end{enumerate}
\end{assum}

\begin{remark}
Note that by requiring that $\mu > 1$, we are assuming that the graph is in the supercritical regime, where
with high probability there exists a unique strongly connected component of linear size, see \cite{Cooper2004,Penrose2014}. In this regime, the probability that there exists a directed path between the two randomly chosen nodes is asymptotically positive, while it is a rare event in both the critical ($\mu = 1$) and subcritical ($\mu < 1$) cases.
\end{remark}

To provide some insights into these assumptions and relate them to the construction of the coupling in Section~\ref{S.Coupling}, it is useful to define first the Kantorovich-Rubinstein distance (also known as Wasserstein metric of order one), which is a metric on the space of probability measures. In particular, convergence in this sense is equivalent to weak convergence plus convergence of the first absolute moments. 

\begin{defn} \label{d.wasserstein}
    Let $M(\mu, \nu)$ denote the set of joint probability measures on $\mathbb R \times \mathbb R$  with marginals $\mu$ and $\nu$. Then, the Kantorovich-Rubinstein distance between $\mu$ and $\nu$ is given by
$$d_1(\mu, \nu) = \inf_{\pi \in M(\mu, \nu)} \int_{\mathbb R\times \mathbb R} | x - y | \, d \pi(x, y).$$
\end{defn}

We point out that $d_1$ is only strictly speaking a distance when both $\mu$ and $\nu$ have finite first absolute moments. Moreover, it is well known that
  \begin{equation*}
  d_1(\mu, \nu) = \int_{0}^1 | F^{-1}(u) - G^{-1}(u) | du = \int_{-\infty}^\infty | F(x) - G(x) | dx,
  \end{equation*}
where $F$ and $G$ are the cumulative distribution functions of $\mu$ and $\nu$, respectively, and $f^{-1}(t) = \inf\{ x \in \mathbb{R}: f(x) \geq t\}$ denotes the pseudo-inverse of $f$. It follows that the optimal coupling of two real random variables $X$ and $Y$ is given by $(X, Y) = (F^{-1}(U), G^{-1}(U))$, where $U$ is uniformly distributed in $[0, 1]$.

With some abuse of notation, for two distribution functions $F$ and $G$ we use $d_1(F,G)$ to denote the Kantorovich-Rubinstein distance between their corresponding probability measures. We refer the interested reader to \cite{Villani2008} for more details. 

\begin{remark} \label{R.LimitMoments} \begin{itemize} 
\item [(i)] In terms of the previous definition, the first condition in Assumption \ref{A.Wasserstein} can also be written as 
\[
	d_1(G_n^\pm, G^\pm) \leq n^{-\varepsilon} \qquad \text{and} \qquad d_1(F_n^\pm, F^\pm) \leq 
	n^{-\varepsilon}.
\]
Furthermore, since  
\[
	\nu_n = \frac{L_n}{n} \qquad \text{and} \qquad \mu_n = \frac{1}{L_n} \sum_{r=1}^n D_r^- D_r^+
\]
are the common means of $g_n^+, g_n^-$, and $f_n^+, f_n^-$, respectively, it follows from Definition \ref{d.wasserstein}  that 
\[
	\left| \nu_n - \nu \right| \leq n^{-\varepsilon} \qquad \text{and} 
	\qquad \left| \mu_n - \mu \right| \leq n^{-\varepsilon}.
\]
 Hence, the first set of assumptions simply state that the empirical degree distributions and the empirical size-biased degree distributions converge weakly, along with their means. 
 
\item[(ii)] The second condition in Assumption \ref{A.Wasserstein} implies that
\[
	\sum_{i=0}^\infty i^{1+\kappa} f^\pm(i) \leq \liminf_{n \to \infty} \sum_{i=0}^\infty 
	i^{1+\kappa} f^\pm_n(i) \leq K_\kappa/\nu < \infty,
\]
i.e., $f^\pm$ has finite moments of order $1+\kappa$. 

\item[(iii)] We point out that any bi-degree sequence satisfying Assumption \ref{A.Wasserstein} will also be ``proper" in the sense of \cite{Cooper2004}, provided that $g^+$ and $g^-$ have finite variance and that the maximum degree is smaller or equal than $n^{1/2}/\log n$. Hence, under these additional conditions, the results in \cite{Cooper2004} regarding the bow-tie structure of the supercritical directed configuration model hold. 
\end{itemize}
\end{remark}

Since, as mentioned earlier, the bi-degree sequence $({\bf D}_n^-, {\bf D}_n^+)$ may itself be generated through a random 
process, we only require that Assumption \ref{A.Wasserstein} holds with high probability. More precisely, if we let 
\begin{align*}
\Omega_n &= \left\{ \max\left\{ d_1(G_n^+, G^+), d_1(G_n^-, G^-), d_1(F_n^+, F^+), d_1(F_n^-, F^-) \right\} \leq n^{-\varepsilon} \right\} \\
&\hspace{5mm} \cap \left\{  \sum_{r=1}^n ((D_r^-)^\kappa + (D_r^+)^\kappa) D_r^+ D_r^- \leq K_\kappa n \right\},
\end{align*}
then our condition will be that $P(\Omega_n) \to 1$ as $n \to \infty$. In Section \ref{SS.IIDAlgorithm} we show that the {\em I.I.D. Algorithm} presented there satisfies this condition.

\subsection{Main result} \label{SS.MainResult}

Our main result, Theorem~\ref{thm:main_result} below, establishes that the distance between two randomly chosen nodes grows logarithmically in the size of the graph, and characterizes the spread around the logarithmic term. 

In the statement of our results, we use $H_n$ to denote the hopcount, or distance, between two randomly chosen nodes in a graph of size $n$. Since the graph is directed, we say that the hopcount between node $i$ and node $j$ is $k$ if there exists a directed path of length $k$ from $i$ to $j$; if there is no directed path from $i$ to $j$ we say that the hopcount is infinite. Since the two nodes are chosen at random, we can assume without loss of generality that $H_n$ is the hopcount from the first node to the second one. 

The last thing we need to do before stating Theorem~\ref{thm:main_result} is to introduce the limiting random variables appearing in the characterization of the hopcount. To this end, let $g^\pm$ and $f^\pm$ be the probability mass functions from Assumption~\ref{A.Wasserstein}. Throughout the paper we will use $\{\hat{Z}^\pm_k: k \ge 0 \}$, $\hat{Z}^\pm_0 = 1$, to denote a delayed Galton-Watson process where nodes in the tree have offspring according to distribution $f^\pm$, with the exception of the root node which has a number of offspring distributed according to $g^\pm$. Note that $W^\pm_k = \hat Z_k^\pm /(\nu \mu^{k-1})$ is a mean one martingale with respect to the filtration generated by the process $\{\hat Z_k^\pm : k \geq 1\}$. 
Hence, by the martingale convergence theorem, 
\[
	W^\pm = \lim_{k \to \infty} \hat Z_k^\pm/(\nu \mu^{k-1}) \quad \text{a.s.}
\]
exists and satisfies $E[W^\pm] \leq 1$. 

To see that under Assumption~\ref{A.Wasserstein} $W^+$ and $W^-$ are non-trivial, it is useful to 
define first  $\{ \mathcal{Z}_k^\pm: k \geq 0\}$ to be a (non-delayed) Galton-Watson process having 
offspring distribution $f^\pm$ and let
\[
	\mathcal{W}^\pm = \lim_{k \to \infty} \mathcal{Z}_k^\pm/\mu^k \quad \text{a.s.}
\]
be its corresponding martingale limit. Now recall from Remark \ref{R.LimitMoments} that $f^+$ and $f^-$ have finite moments of order 
$1+\kappa > 1$, which implies that $\sum_{j=1}^\infty j \log j f^\pm(j) < \infty$, a necessary and 
sufficient condition for $\mathcal{W}^\pm$ to be non-trivial (see, e.g., \cite{Athreya2012}). More 
precisely, if $q^\pm = P( \mathcal{Z}^\pm_m = 0 \text{ for some $m$})$ denotes the probability of 
extinction of $\{\mathcal{Z}_k^\pm: k \geq 0\}$, then $q^\pm < 1$ and $P(\mathcal{W}^\pm = 0) = 
q^\pm$. It follows from these observations that $E[ W^\pm ] = 1$ (Lemma \ref{L.Extinction} contains an expression for $P(W^\pm = 0)$). 

We are now ready to state the main result of the paper; $\lfloor x \rfloor$ denotes the largest integer smaller or equal to $x$. 

\begin{thm}\label{thm:main_result}
	Let $\{ \mathcal{G}_n: n \geq 1\}$ be a sequence of graphs generated through the DCM from a sequence of 
	bi-degree sequences $\{({\bf D}_n^-, {\bf D}_n^+): n \geq 1\}$ 	satisfying $P(\Omega_n) \to 1$ as 
	$n \to \infty$. Let $H_n$ denote the hopcount between two randomly chosen nodes in $\mathcal{G}_n$. Then, 
	there exist random variables $\{\mathcal{H}_n\}_{n  \in \mathbb{N}}$ such that for each (fixed) 
	$t \in \Z$,
	\begin{equation}
		\lim_{n \to \infty} \left|\CProb{H_n - \floor{\log_\mu n} = t}{H_n < \infty} 
		- \Prob{\mathcal{H}_{n} = t}\right| = 0,
	\label{eq:main_result}
	\end{equation}
	where $\mathcal{H}_n$ has distribution
	\begin{equation}
	\Prob{\mathcal{H}_n \leq x} = 1 - E\left[ \left. \exp\left\{-\frac{\nu}{\mu - 1} \cdot
	\frac{\mu^{\floor{\log_\mu n} + \lfloor x \rfloor}}{n} \limWin \limWout \right\} \right| \limWin \limWout > 0 
	\right], \qquad x \in \mathbb{R}.
\label{eq:main_result_R}
\end{equation}
\end{thm}

As a straightforward corollary, we obtain the asymptotic equivalence of $H_n$ and $\log_\mu n$ in probability.
\begin{cor}
Under the same assumptions as Theorem~\ref{thm:main_result}, and for any $\epsilon > 0$, 
$$\lim_{n \to \infty} P\left( \left. \left| \frac{H_n}{\log_\mu n} - 1 \right| > \epsilon \right| H_n < \infty \right) = 0.$$
\end{cor}

Theorem \ref{thm:main_result} shows that the directed distance between two randomly chosen nodes in the DCM scales logarithmically in the size of the graph, which is consistent with existing results for the undirected configuration model (CM) under the assumption that the degree distribution has finite variance \cite{Hofstad2005}. We remark that no assumption is made concerning the simplicity of the graph, since the hopcount is unaffected by the existence of multiple edges and self-loops. For instance, removing all self-loops and merging all duplicate edges into a single edge, as is done in the erased configuration model, will not change the hopcount.

To understand the result of Theorem \ref{thm:main_result}, including the appearance of the martingale limits $W^\pm$, note that directed graphs with high connectivity
consist of a strongly connected component (SCC), a set of nodes with directed paths going into the SCC (the inbound wing), a set of nodes with directed paths exiting the SCC (the outbound wing), and
some additional secondary structures. This, so-called, bow-tie structure has been observed experimentally in the web graph \cite{Broder2000}, and has been established for the supercritical directed configuration model in \cite{Cooper2004}; see also \cite{Timar2016} for a more detailed analysis of the secondary structures. More precisely, the work in \cite{Cooper2004}, shows that the inbound wing consists of nodes whose in-component is of linear size but whose out-component is small (i.e., of order $o(n)$), the outbound wing consists of nodes whose out-component is of linear size but whose in-component is small, and the SCC is the set of nodes having both linear size in-component and linear size out-component. The branching processes $\hat Z_k^+$ and $\hat Z_k^-$ describe the breadth-first exploration process of the out-component of the first randomly chosen node and the in-component of the second one, respectively, whose sizes are approximately $W^+ \nu \mu^{k-1}$ and $W^- \nu \mu^{k-1}$. We refer the reader to
\cite{Hofstad2005a} (pp. 712-714) for a more detailed explanation relating the hopcount with the branching processes appearing in the limit.

The interesting difference between the directed and undirected cases lies in the observation that Assumption~\ref{A.Wasserstein} can hold with high probability for degree sequences having infinite variance (as shown in Section \ref{SS.IIDAlgorithm}), hence showing that the distance remains logarithmic even when in its undirected counterpart becomes of order $\log\log n$ \cite{Hofstad2005a}. To explain this, note that distances in the CM get smaller as the degree distribution gets heavier (i.e., more variable) presumably because of the appearance of nodes with extremely large degrees, that should create shortcuts between nodes in their connected component\footnote{The role that high degree nodes play in the creation of shortcuts is best understood through the notion of betweenness centrality, which computes the fraction of shortest paths that go through a given node. For the undirected configuration model, it was shown numerically in \cite{Goh2003} that the betweenness centrality is positively correlated with the degree, which is consistent with our intuitive explanation of why distances get smaller the more spread out the degree distribution becomes.}. In contrast, when the graph is directed, increasing the variability of the in- and out-degree distributions does not necessarily imply the appearance of more shortcuts, since even if there are more nodes with very large in-degrees or very large out-degrees, they may not be the same nodes, e.g., when the in-degree is independent of the out-degree, it is unlikely that a node has both large in-degree and large out-degree. Our results are consistent with the intuition that if the nodes with very large in-degrees are the same as those with very large out-degrees (i.e., positively correlated in- and out-degrees), then more shortcuts should be created and the distances will get smaller.

To complement the main theorem we also compute the asymptotic probability that the hopcount is finite, which can be expressed in terms of the survival properties of the delayed branching processes $\{\hat{Z}^+_k: k 
\geq 1\}$ and $\{\hat Z^-_k: k \geq 1\}$.

\begin{prop}\label{prop:finite_paths}
	Let $\{ \mathcal{G}_n: n \geq 1\}$ be a sequence of graphs generated through the DCM from a sequence of 
	bi-degree sequences $\{({\bf D}_n^-, {\bf D}_n^+): n \geq 1\}$ 	satisfying $P(\Omega_n) \to 1$ as 
	$n \to \infty$ let $H_n$ denote the hopcount between two randomly chosen nodes in $\mathcal{G}_n$. Then,
	\begin{equation}
		\lim_{n \to \infty} \Prob{H_n < \infty} = s^+ s^-,
	\label{eq:prob_path_exists}
	\end{equation}
	where $s^\pm = P(W^\pm > 0)$.
\end{prop}

To provide some insights into this probability, we refer again to the bow-tie structure of the supercritical directed configuration model, where $S$ is the SCC, $K^-$ is the inbound wing, and $K^+$ is the outbound wing. As the work in \cite{Cooper2004} shows, if we let $L^-$ and $L^+$ denote the set of nodes with in-component, respectively out-component, of linear size, then $S = L^- \cap L^+$, $K^- = L^- \cap (L^+)^c$ and $K^+ = L^+ \cap (L^-)^c$. Moreover, the proof of our main coupling result (Theorem \ref{thm:main_coupling}) shows that  $s^+$  ($s^-$) is the asymptotic probability that a randomly chosen node in the graph belongs to $L^+$ ($L^-$), which is consistent with Theorem~1.2\footnote{See Remark~\ref{R.LimitMoments}(iii).} in \cite{Cooper2004}, suggesting that the bow-tie structure proved there may hold even under the weaker assumptions of this paper.

With respect to the martingale limits $W^+$ and $W^-$ appearing in \eqref{eq:main_result_R}, we point out that although it is in general difficult to compute them analytically, it can easily be done numerically, e.g., by using the Population Dynamics algorithm described in \cite{Chen2014}. We use this algorithm in Section \ref{S.Examples} below for validating our theoretical results.


\section{Construction of a bi-degree sequence and numerical examples}\label{S.Examples}

To illustrate the accuracy of the approximation for the hopcount between two randomly chosen nodes provided by Theorem~\ref{thm:main_result}, we give in this section several numerical examples for different choices of the bi-degree sequence. This requires us to construct a sequence of bi-degree sequences $\left\{({\bf D}_n^-, {\bf D}_n^+) : n \ge 1\right\}$ satisfying Assumption \ref{A.Wasserstein}  with high probability for some prescribed joint distribution for the in- and out-degrees. As pointed out earlier, there are many ways of constructing such sequences, but for the sake of completeness, we include here an algorithm based on i.i.d.~samples from the prescribed degree distribution.

\subsection{The i.i.d.~algorithm}\label{SS.IIDAlgorithm}

Let $G(x,y)$ be a joint distribution function on $\mathbb{N}^2$ such that if $(\mathscr{D}^-, \mathscr{D}^+)$ is distributed according to $G$, then $E[\mathscr{D}^- ] = E[\mathscr{D}^+]$, $E[ |\mathscr{D}^- - \mathscr{D}^+|^{1+\kappa}] < \infty$ and $E[ (\mathscr{D}^-)^{1+\kappa} (\mathscr{D}^+)^{1+\kappa}] < \infty$ for some $0 < \kappa \leq 1$. Set $\delta = c\kappa/(1+\kappa)$, for some $0 < c < 1$ if $\kappa < 1$ or choose any $0 < \delta < 1/2$ if $\kappa = 1$. 
\begin{list}{}{\leftmargin 10mm}
\item[{\sc Step 1:}] Sample $\{ (\mathscr{D}^-_i, \mathscr{D}^+_i) \}_{i=1}^n$ as i.i.d. vectors distributed according to $G(x,y)$. 
\item[{\sc Step 2:}] Define $\Delta_n = \sum_{i=1}^n (\mathscr{D}^-_i - \mathscr{D}^+_i)$. If $|\Delta_n| \leq n^{1-\delta}$, proceed to {\sc Step 3}; else, repeat {\sc Step 1}.
\item[{\sc Step 3:}] Select $|\Delta_n|$ indices from $\{1, 2, \dots, n\}$ uniformly at random (without replacement) and set
$$D_i^- = \mathscr{D}^-_i + \tau_i \qquad \text{and} \qquad D_i^+ = \mathscr{D}^+_i + \chi_i, \qquad i = 1, 2, \dots, n,$$ 
where
$$\tau_i = 1( \text{$\Delta_n \leq 0$ and $i$ was selected}) \qquad \text{and} \qquad  \chi_i = 1(\text{$\Delta_n > 0$ and $i$ was selected}).$$ 
\end{list}

This algorithm was first introduced in \cite{chen2013} for the special case where $\mathscr{D}^-$ and $\mathscr{D}^+$ are independent. There, it was shown that the degree sequences generated by the algorithm are graphical w.h.p., i.e., they can be used to construct simple graphs. Moreover, the empirical joint distribution of the degrees in a simple graph generated through either the {\em repeated} DCM or the {\em erased} DCM, converges in probability to $G(x,y)$\footnote{Note that the joint degree distribution of the nodes in the resulting graph is not $G(x,y)$, since this distribution is changed by {\sc Step 1} and {\sc Step 2} of the algorithm, as well as by the pairing process itself.}, see Theorems 2.3 and 2.4 in \cite{chen2013}\footnote{These theorems are stated for the case when $\mathscr{D}^+$ and $\mathscr{D}^-$ are independent, but a close look at the proofs shows that they remain valid when they are dependent.}.

Note that the $\{\mathscr{D}_i^- - \mathscr{D}_i^+\}$ are zero-mean random variables with $\Exp{|\mathscr{D}^- - \mathscr{D}^+|^{1 + \kappa}} < \infty$. Hence, using Burkholder's inequality (see Lemma \ref{L.Burkholder}) one obtains that
\[
	\Prob{|\Delta_n| > n^{1 - \delta}} = O\left(n^{1 - (1 + \kappa)(1 - \delta)}\right) = O\left(n^{-(1 - c)\kappa}\right),
\] 
for $\kappa < 1$, while it is $O(n^{2\delta - 1})$ when $\kappa = 1$. Hence the probability of success in {\sc Step 2} is $1 - O(n^{-a})$ for some $a > 0$.

We also point out that the i.i.d.~algorithm only requires the moment conditions $E[ |\mathscr{D}^- - \mathscr{D}^+|^{1+\kappa}] < \infty$ and $E[ (\mathscr{D}^-)^{1+\kappa} (\mathscr{D}^+)^{1+\kappa}] < \infty$, and therefore can be used to generate any light-tailed degree sequence as well as the vast majority of scale-free (heavy-tailed) degree distributions. It also includes as a special case the $d$-regular bi-degree sequence. 

The following result shows that the degree sequences generated by this algorithm satisfy Assumption~\ref{A.Wasserstein} with high probability. 

\begin{thm} \label{T.IIDAlgorithm}
Let $G^-$ denote the marginal distribution of $\mathscr{D}^-$ and $G^+$ denote that of $\mathscr{D}^+$; define $F^+(x) = E[1(\mathscr{D}^+ \leq x) \mathscr{D}^-] /\nu$ and $F^-(x) = E[1(\mathscr{D}^- \leq x) \mathscr{D}^+] /\nu$, where $\nu = E[\mathscr{D}^+] = E[\mathscr{D}^-]$. Then, for any $0 < \varepsilon < \delta$ and $E[ (\mathscr{D}^-)^{1+\kappa} \mathscr{D}^+ + \mathscr{D}^- (\mathscr{D}^+)^{1+\kappa} ]  < K_{\kappa} < \infty$, we have
$$\lim_{n \to \infty} P(\Omega_n) = 1.$$
\end{thm}

\subsection{The hopcount distribution}

In order to compute the hopcount distribution, we constructed 20 graphs of size $n = 10^6$, using 
the DCM for different choices of bi-degree sequence. For each of these graphs
we computed the neighborhood function, which gives for each $t > 0$ the number of pairs of nodes
at distance at least $t$. For the computation of the neighborhood function we used the HyperBall 
algorithm \cite{Boldi2013}, which is part of the Webgraph Framework \cite{Boldi2004}. We used 
HyperBall since it implements the HyperANF algorithm \cite{Boldi2011a}, which is designed to give 
a tight approximation of the neighborhood function of large graphs. From the neighborhood function 
we determined, for all finite $t$, the number of shortest paths of length $t$. In this way, we 
compute the distance between all pairs of nodes, with finite distance, in 20 independently 
generated graphs. We then took the empirical distribution of these values as a approximation of the hopcount 
distribution. 

We point out that since $H_n$ was defined as the hopcount between two randomly selected nodes, the natural unbiased estimator for the distribution of $H_n$ is the one obtained from randomly selecting pairs of nodes in independent graphs and using the corresponding empirical distribution function. However, this approach is computationally too intensive considering the amount of effort needed to generate one graph. Our approach is considerably more efficient, and although the empirical distribution function it generates does not consist of i.i.d.~samples (samples from the same graph are positively correlated), it produces results that are in close agreement with the theoretical approximation in Theorem~\ref{thm:main_result}. Additional experiments not included in this paper showed that the two approaches produce similar results, with the method used in this paper exhibiting smaller variance. 

The three examples below illustrate the accuracy of the approximation provided by Theorem~\ref{thm:main_result} for different choices of bi-degree sequences. All three examples are special cases of the i.i.d.~algorithm, and thus satisfy Assumption~\ref{A.Wasserstein}. 

\subsubsection{$d$-regular bi-degree sequence}
A d-regular bi-degree sequence satisfies $D_i^+ = d = D_i^-$ for all $1 \le i \le n$. It readily 
follows that the probability densities $g^\pm$ and $f^\pm$ have just one atom at $d$. Moreover, we have 
$\hat{Z}^\pm_k = d^k = \mu^k$ for all $k \ge 1$, hence $W^\pm = 1$ and,
\begin{equation}
	\Prob{\mathcal{H}_n \leq x} = 1 - \exp\left\{-\frac{d^{\floor{\log_d n} + \lfloor x \rfloor }}{(d - 1)n}\right\}, \qquad x \in \mathbb{R}.
	\label{E.DRegularTailProbability}
\end{equation}

\begin{figure} 
	\centering \vspace{-2cm}
	\begin{subfigure}{0.5\linewidth}
		\includegraphics[scale=0.45]{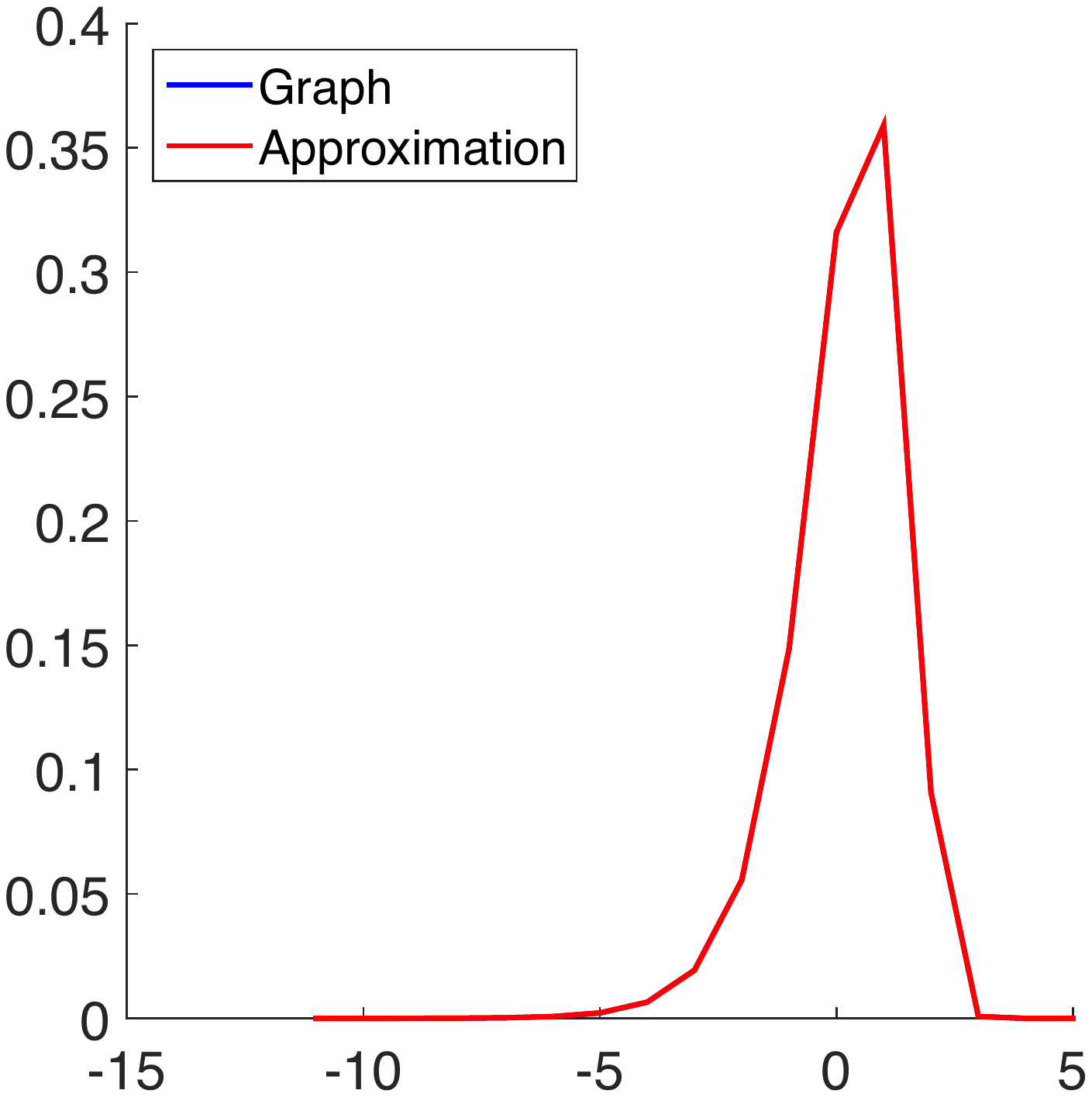}
		\caption{}
		\label{SF.DRegularPDF}
	\end{subfigure}~
	\begin{subfigure}{0.5\linewidth}
		\includegraphics[scale=0.45]{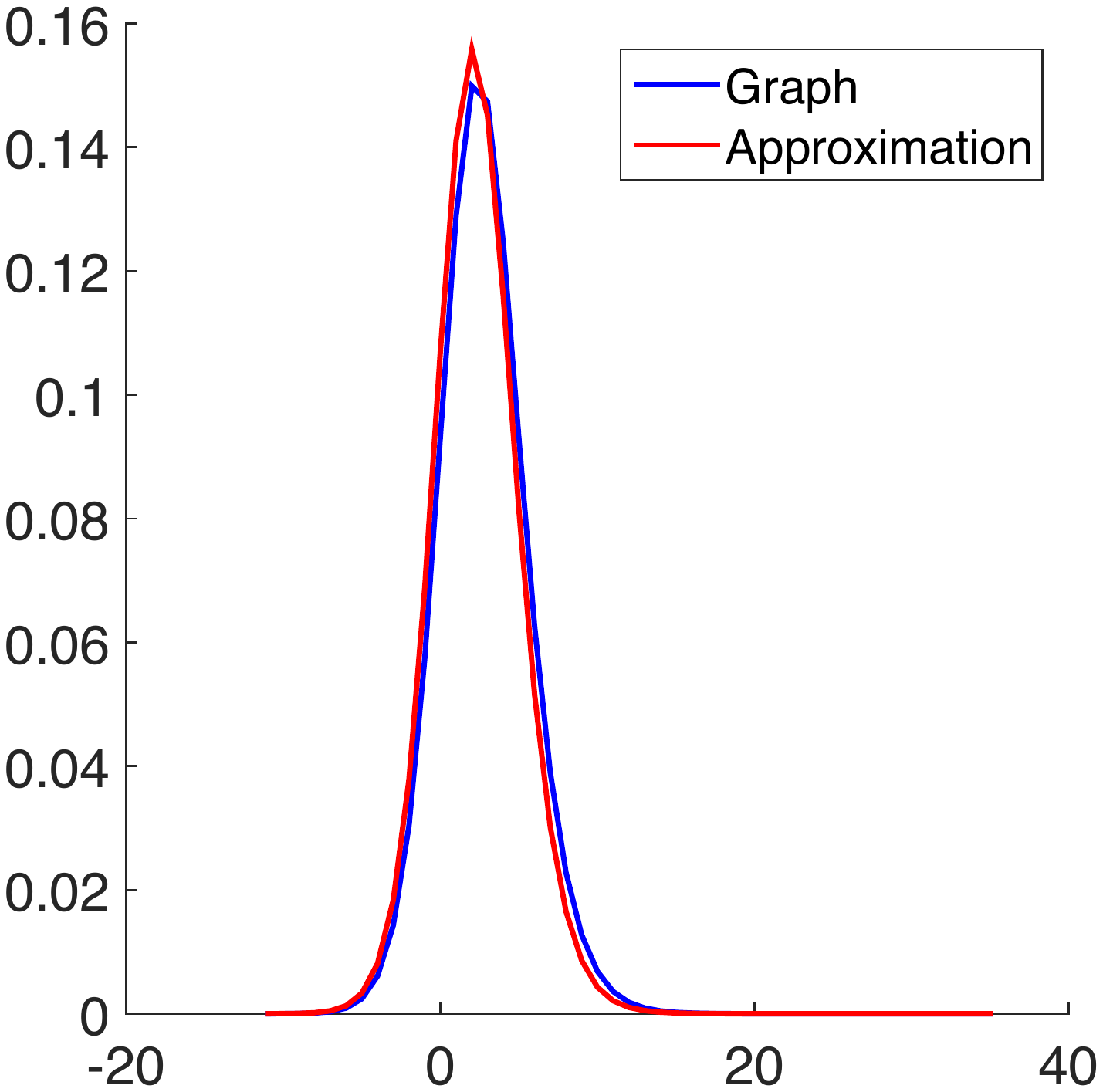}
		\caption{}
		\label{SF.IIDSequencePDF}
	\end{subfigure} \\ \vspace{-5cm}
	\begin{subfigure}{0.5 \linewidth}
		\includegraphics[scale=0.45]{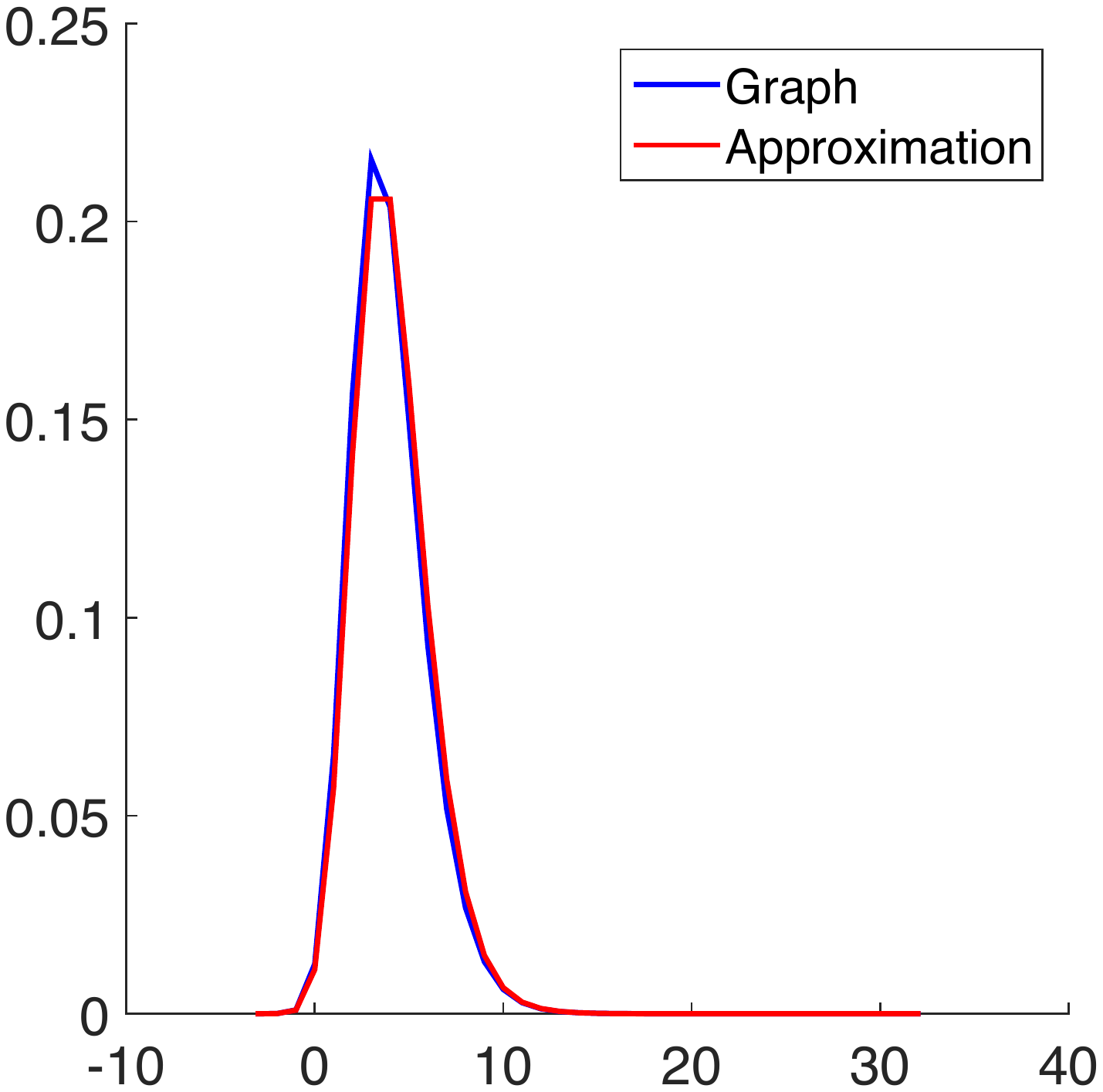}
		\caption{}
		\label{SF.DSequencePDF}
	\end{subfigure}
	\caption{Hopcount probability mass function compared to the approximation provided by Theorem~\ref{thm:main_result} for: (a) a $3$-regular bi-degree sequence; (b) a bi-degree sequence generated by the i.i.d.~algorithm with independent in- and out-degrees; and (c) a bi-degree sequence generated by the i.i.d.~algorithm with dependent in- and out-degrees. The Kolmogorov-Smirnov distance in each case is: (a) $1.3\times 10^{-4}$, (b) 0.0583, and (c) 0.0353. In all cases the graphs had $n = 10^6$ nodes.}
	\label{F.PDFcomparison}
\end{figure}

In Figure~\ref{F.PDFcomparison}(a) we plotted the probability mass functions of both the hopcount
distribution and that of its theoretical limit~\eqref{E.DRegularTailProbability}. The plots are 
indistinguishable in the figure, with a Kolmogorov-Smirnov distance of $1.3 \times 10^{-4}$. This 
shows that for non-random sequences, the approximation provided by Theorem \ref{thm:main_result} is almost exact.

\subsubsection{I.I.D.~bi-degree sequence with independent in- and out-degrees}

Following the result from Theorem~\ref{T.IIDAlgorithm}, we computed the hopcount distribution for
bi-degree sequences, generated by the i.i.d.~algorithm, using as the in- and out-degree distributions Poisson mixed with Pareto rates, and keeping the in-degree and out-degree independent of each other. More precisely, we chose $\Lambda_1$ and $\Lambda_2$ to be independent Pareto 
random variables, both with scale parameter $1$ and shape parameter $3/2$, and then set $\mathscr{D}^-$ and 
$\mathscr{D}^+$ be i.i.d.~with conditional distributions
\[
	P\left(\left. \mathscr{D}^- = k \right| \Lambda_1 = \lambda \right) = P\left( \left. \mathscr{D}^+ = k \right| \Lambda_2 = \lambda \right) =  \frac{\lambda^k e^{-\lambda}}{k!} \qquad k = 0, 1, 2, \dots.
\]
It can be verified that (see Proposition~8.4 in \cite{Grandell1997}) that
$$P(\mathscr{D}^- \geq k) \sim c_1 k^{-3/2} \qquad \text{and} \qquad P(\mathscr{D}^+ \geq k) \sim c_2 k^{-3/2},$$
as $k \to \infty$, for some constants $c_1, c_2 > 0$.

Note that the independence between $\mathscr{D}^-$ and $\mathscr{D}^+$ implies that the size-biased distributions $f^+$ and $f^-$ are equal to the unbiased ones, i.e., $f^\pm = g^\pm$. Hence, $\mu = \nu = 3$ and the branching processes $\{\hat{Z}^+_k : k \ge 1\}$ and $\{\hat{Z}^-_k : k \ge 1\}$ are not delayed.

In order to compute our theoretical approximation for the hopcount we also need to compute $\Prob{\mathcal{H}_n > k}$, which is written in terms of $W^+$ and $W^-$. Since $W^+$ and $W^-$ are not known in general, we estimate them numerically using the approach from \cite{Chen2014}, which describes a bootstrap algorithm for simulating the endogenous solutions of branching linear recursions. For 
this we first observe that $W^+$ and $W^-$ satisfy the following stochastic fixed-point equations: 
\[
	W^- \stackrel{d}{=} \sum_{i = 1}^{\mathscr{D}^-} \frac{W_i^-}{\mu} \qquad \text{and} \qquad
	W^+ \stackrel{d}{=} \sum_{i = 1}^{\mathscr{D}^+} \frac{W_i^+}{\mu},
\]
where $W_i^\pm$ are i.i.d.~copies of $W^\pm$, independent of $\mathscr{D}^-$ and $\mathscr{D}^+$. Using the algorithm in \cite{Chen2014} for 30 generations of the trees with a sample pool of size $10^6$, we obtained $10^6$ observations for each of $W^+$ and $W^-$, with the sample for $W^+$ independent of that for $W^-$. We then used these samples to estimate
\[
	\CExp{\exp\left\{-\frac{\nu}{\mu-1} \cdot \frac{\mu^{\lfloor \log_\mu n \rfloor + k}}{n}  W^+ W^- \right\}}{W^+W^- > 0}
	\quad \text{for } k = 0, 1, \dots.
\]

The results for the hopcount distribution are shown in Figure~\ref{F.PDFcomparison}(b). The 
Kolmogorov-Smirnov distance in this case is $0.0583$.

\subsubsection{I.I.D.~bi-degree sequence with dependent in- and out-degrees}

Our third and last example is for a bi-degree sequence obtained using the i.i.d.~algorithm but for the case where $\mathscr{D}^-$ and $\mathscr{D}^+$ are dependent. We take the extreme case where $D_i^-
= D_i^+$ for all $1 \le i \le n$. To obtain such a sequence we generate the $D_i^-$ by sampling 
from a Zipf distribution with corpus size $10^3$ and exponent $7/2$ and set $D_i^+ = D_i^-$, that is
\[
	P(\mathscr{D}^+ = t) = t^{-7/2}/\zeta(7/2) \quad \text{for all } t = 1, 2, \dots,
\]
where $\zeta(s)$ is the Riemann zeta function. Observe
that since the exponent is larger than $3$, the distribution has finite $2 + \varepsilon$ moment, 
for $0 < \varepsilon < 1/2$. Therefore, it follows from Theorem~\ref{T.IIDAlgorithm} that this bi-degree sequence satisfies Assumption \ref{A.Wasserstein} with high probability. We used a Zipf distribution here since then, the 
sized-bias distribution will again be Zipf, with exponent $5/2$.

The $W^+$ and $W^-$ were again simulated using the algorithm in \cite{Chen2014} with the same number of generations and the same pool size as for the independent case above, but with the appropriate sized-biased distribution and the corresponding delay for the first generation of the tree. 

The results for the hopcount are shown in Figure \ref{F.PDFcomparison}(c), and the Kolmogorov-Smirnov distance is $0.0353$. 




\section{Coupling with a branching process}
\label{S.Coupling}

Given a directed graph $\mathcal{G}_n$ of size $n$ the shortest directed path from node $v_1$ to node $v_2$ can be computed by starting two breadth-first exploration processes, one to uncover the out-component of  $v_1$, call this $B^+(v_1)$, and another one to uncover the in-component of $v_2$, call it $B^-(v_2)$. If $B^+(v_1) \cap B^-(v_2) \neq \varnothing$, then there exists a finite $(v_1,v_2)$-path, whereas if this intersection is always empty, there is none. We point out that since shortest paths do not contain cycles, the exploration of the components, either inbound or outbound, requires only that we keep track of edges with nodes not previously uncovered.

The first step in proving Theorem \ref{thm:main_result} is to couple the breadth-first exploration processes described above, starting from uniformly chosen nodes in $\mathcal{G}_n$, with two independent branching process. This is a well known approach for analyzing the properties of random graphs, also referred to as a branching process argument.

The main result of this section is Theorem~\ref{thm:main_coupling}, along with its more immediately 
useful corollary (Corollary~\ref{C.MainCoupling}), which is the key ingredient in
the proof of Theorem \ref{thm:main_result}. 

\subsection{Exploration of new stubs}

Similarly to the construction in \cite{Hofstad2005}, we start by designating all the $n$ nodes as {\em inactive}, meaning they have not been uncovered yet, and setting $Z_0^\pm = 1$ (note that in 
\cite{Hofstad2005} it is the stubs themselves that are labeled, not the nodes).  Let $\emptyset$ denote the fictional first stub, and set $A_0^\pm =\{\emptyset\}$; call this initialization step 0. The 
process $\{Z_k^\pm : k \geq 0\}$ will keep track of the number of outbound (inbound) stubs discovered during 
the $k$-th step of the exploration process, as we will now describe. The superscript $\pm$ refers 
to whether the exploration follows the outbound stubs (for which we use the superscript $+$), or the inbound stubs (for which we use the superscript $-$).

In step 1 we randomly select a node and set $Z_1^\pm = j$ if it has $j$ outbound (inbound) stubs; 
we set its state to {\em active}, meaning it has already been uncovered. To identify each of
the outbound (inbound) stubs we index them 1 through $j$ and let $A_1^\pm = \{1, \dots, j\}$
be the set of the indices of the newly discovered stubs. For the second step of the exploration 
process we will need to traverse all $Z_1^\pm$ outbound (inbound) stubs, which we do 
sequentially and in lexicographic order with respect to their indexes. Here, we say that we have traversed an outbound (inbound) stub if we have identified the node it leads to and discovered how many outbound (inbound) stubs this new node has. If the stub is pointing to an {\em inactive} node, we label the node as {\em active}, index all its outbound (inbound) stubs with a name of the form $(i,j)$, $j \geq 1$, and then proceed to explore the next outbound (inbound) stub. If the stub is pointing to an {\em active} node no new outbound (inbound) stubs are discovered. Once we
are done exploring all $Z_1^\pm$ outbound (inbound) stubs we set $Z_2^\pm$ to be the number of newly discovered  outbound (inbound) stubs and let $A_2^\pm$ denote the set of their indices.

In general, in step $k$ we will traverse all $Z_{k-1}^\pm$ outbound (inbound) stubs, in lexicographic 
order, discovering new nodes and hence new outbound (inbound) stubs. If outbound (inbound) stub 
${\bf i} = (i_1, \dots, i_{k-1})$ is paired with an inbound (outbound) stub belonging to an {\em inactive} node,
then the outbound (inbound) stubs of the newly discovered node receive an index of the form $(i_1, \dots, i_{k-1}, 
i_{k})$, $i_{k} \geq 1$; if outbound (inbound) stub {\bf i} is paired with an inbound (outbound) stub belonging to an {\em active} node, then no new outbound (inbound) stubs are discovered. Once we have traversed all $Z_{k-1}^\pm$ outbound (inbound) stubs we set $Z_{k}^\pm$ to be the number of new outbound (inbound) stubs discovered in 
step $k$. The process continues until all $L_n$ outbound (inbound) stubs have been traversed. 

Note that the process $\{Z_k^\pm: k \geq 0\}$ defines a labeled tree, where the ``individuals" are 
the outbound (inbound) stubs discovered in step $k$ ($Z^\pm_0 = 1$), not the nodes of the graph 
themselves. In addition to keeping track of $Z_k^\pm$, we will also keep track of ``time" in the 
exploration process, where time $t$ means we have traversed $t$ outbound (inbound) stubs.

\subsection{Construction of the coupling} \label{SS.Coupling}

To study the distance between two randomly chosen nodes we will couple the exploration of the graph 
described above with a branching process. To do this we first note that the exploration process 
 is equivalent to assigning to outbound stub ${\bf i} \neq \emptyset$ a number of offspring $\chi_{\bf i}^+$ 
chosen according to the (random) probability mass function
\begin{equation}
	h^+_{\bf i}(t) = 
	\begin{cases}
		\frac{1}{L_n - T^+_{\bf i}} \sum_{r=1}^n 1(D_r^+ = t) D_r^- \mathcal{I}_r(T_{\bf i}^+) , 
		& t = 1, 2, \dots, \\
		\frac{1}{L_n - T^+_{\bf i}}  \left\{ \sum_{r=1}^n 1(D_r^+ = 0) D_r^- \mathcal{I}_r(T_{\bf i}^+) 
		+ V_{\bf i}^- \right\}, & t = 0,
	\end{cases}
\end{equation}
where $T_{\bf i}^+$ is the number of outbound stubs that have been traversed up until the moment outbound stub {\bf i} is about to be traversed, $\mathcal{I}_r(t) = 1(\text{node $r$ is {\em inactive} after 
having traversed $t$ stubs})$, and 
\[
	V_{\bf i}^- = L_n - \sum_{r=1}^n D_r^- \mathcal{I}_r(T^+_{\bf i}) - T^+_{\bf i}
\]
is the number of unexplored inbound stubs belonging to active nodes at time $T_{\bf i}^+$. Note that $T_{\bf i}^+$ is also the number of inbound stubs that already belong to edges in the graph up until the moment outbound stub ${\bf i}$ is about to be explored. Symmetrically, we assign to inbound stub ${\bf i}$ a number of offspring 
$\chi_{\bf i}^-$ distributed according to
\begin{equation}
h^-_{\bf i}(t) = \begin{cases}
	\frac{1}{L_n - T^-_{\bf i}} \sum_{r=1}^n 1(D_r^- = t) D_r^+ \mathcal{I}_r(T_{\bf i}^-) , 
	& t = 1, 2, \dots, \\
	\frac{1}{L_n - T^-_{\bf i}} \left\{\sum_{r=1}^n 1(D_r^- = 0) D_r^+ \mathcal{I}_r(T_{\bf i}^-) 
	+ V_{\bf i}^+ \right\}, & t = 0,
\end{cases}
\end{equation}
with $T_{\bf i}^-$ the number of inbound stubs that have been traversed up until the moment inbound stub {\bf i} is about to be explored, and
\[
	V_{\bf i}^+ = L_n - \sum_{r=1}^n D_r^+ \mathcal{I}_r(T^-_{\bf i}) - T^-_{\bf i}
\]
is the number of unexplored outbound stubs belonging to {\em active} nodes at time $T_{\bf i}^-$. As before, we have that $T_{\bf i}^-$ is also the number of outbound stubs that already belong to edges in the graph up until the moment inbound stub ${\bf i}$ is about to be explored. 

Note that the number of outbound (inbound) stubs of the first node, i.e., $Z_1^\pm$, is 
distributed according to $g_n^\pm$. 

The key idea behind the coupling we will construct is that sampling from $h_{\bf i}^\pm$ and sampling from $f_n^\pm$ should be roughly equivalent as long as $T^\pm_{\bf i}$ is not too large. In turn, for large $n$, Assumption~\ref{A.Wasserstein} implies that $f_n^\pm$ is very close to $f^\pm$. It follows that the 
process $\{ Z_k^\pm: k \geq 0\}$ should be very close to a suitably constructed (delayed) branching
process $\{ \hat Z_k^\pm: k \geq 0\}$ having offspring distributions $(g^\pm, f^\pm)$, where $g^\pm$ is the distribution of $\hat Z_1^\pm$ and all other nodes have offspring according to $f^\pm$.

To construct the coupling define $U = \bigcup_{k=0}^\infty \mathbb{N}_+^k$, with the convention 
that $\mathbb{N}_+^0 = \{ \emptyset\}$, and let $\{U_{\bf i}\}_{{\bf i} \in U}$ be a sequence of 
i.i.d.~Uniform$(0,1)$ random variables. For any non-decreasing function $F$ define $F^{-1}(u) = 
\inf\{ x \in \mathbb{R}: F(x) \geq u\}$ to be its pseudo-inverse. Now set the number of outbound 
(inbound) stubs of {\bf i} in the graph to be 
\[
	\chi_{\bf i}^\pm = (H_{\bf i}^\pm)^{-1}(U_{\bf i}), \quad{\bf i} \neq \emptyset, 
	\qquad \chi_\emptyset^\pm = (G_n^\pm)^{-1}(U_\emptyset),
\]
where $H_{\bf i}^\pm$ is the cumulative distribution function of $h_{\bf i}^\pm$, and the number 
of offspring of individual {\bf i} in the outbound (inbound) branching process to be
\[
	\hat \chi_{\bf i}^\pm = (F^\pm)^{-1}(U_{\bf i}), \quad {\bf i} \neq \emptyset, 
	\qquad \hat \chi_\emptyset^\pm = (G^\pm)^{-1}(U_\emptyset).
\]
In addition we let $\hat{A}^\pm_r$ denote the set of individuals in the tree, corresponding to the
process $\{\hat{Z}_k^\pm : k \ge 0\}$, at distance $r$ from the root.

Note that $\chi_{\bf i}$ and $\hat \chi_{\bf i}$ are now coupled through the same $U_{\bf i}$, and in view of the remarks following Definition~\ref{d.wasserstein}, this coupling minimizes the Kantorovich-Rubinstein distance between the distributions $h_{\bf i}^\pm$ and $f^\pm$. Moreover, although the $\chi_{\bf i}^\pm$ are only defined for stubs {\bf i} that have been created 
through the pairing process, the $\hat \chi_{\bf i}^\pm$ are well defined regardless of whether 
{\bf i} belongs to the tree or not. Furthermore, the sequence $\{ U_{\bf i}\}_{{\bf i} \in U}$ defines
the entire branching process $\{ \hat Z_k^\pm: k \geq 0\}$, even after the graph has been fully explored. 

The last thing we need to take care of is the observation that knowing $\chi_{\bf i}^\pm$ in the 
exploration of the graph does not necessarily tell us the identity of the node that stub {\bf i} leads to, since there may be more than one node with $\chi_{\bf i}^\pm$ outbound (inbound) stubs, which is problematic if they do not also have the same number of inbound (outbound) stubs. The construction of the coupling requires that we keep track of both the inbound and outbound stubs discovered when a node first becomes {\em active}, since this information allows us to estimate the remaining number of unexplored stubs. To fix this problem, given $\chi_{\bf i}^\pm = t > 0$, pair outbound (inbound) stub ${\bf i}$ with an inbound (outbound) stub 
randomly chosen from the set of unpaired inbound (outbound) stubs belonging to {\em inactive} nodes 
and having exactly $t$ outbound (inbound) stubs; if $\chi_{\bf i}^\pm = 0$ sample the inbound 
(outbound) stub from the set of unpaired inbound (outbound) stubs belonging to either {\em inactive} or {\em active} nodes having no outbound (inbound) stubs.

Summarizing the notation, we have:
\begin{itemize} \itemsep 0pt
\item $A_r^+$ ($A_r^-$): set of outbound (inbound) stubs created during the $r$th step of the 
exploration process on the graph. 

\item $\hat A_r^+$ ($\hat A_r^-$): set of individuals in the outbound (inbound) tree at distance 
$r$ of the root.

\item $Z_r^+$ ($Z_r^-$): number of outbound (inbound) stubs created during the $r$th step of 
the exploration process.

\item $\hat Z_r^+$ ($\hat Z_r^-$): number of individuals in the $r$th generation of the outbound
(inbound) tree.

\end{itemize}
The main observation upon which the analysis of the coupling is based is that if $|A|$ denotes the cardinality of set $A$, then
\[
	Z^\pm_{k} = \left| A^\pm_{k} \right|  = \left| A^\pm_{k} \cap \hat A^\pm_k \right| 
	+ \left| A^\pm_{k} \cap (\hat A^\pm_k)^c \right|
\]
which implies that
\begin{equation} \label{eq:MainInequality}
	\hat Z^\pm_k - \left| \hat A^\pm_{k} \cap (A^\pm_k)^c \right|  \leq Z^\pm_{k} 
	\leq \hat Z^\pm_k + \left| A^\pm_{k} \cap (\hat A^\pm_k)^c \right| .
\end{equation}

\subsection{Coupling results}

We now present our main result on the coupling between the exploration process $\{Z_k^\pm: k \geq 1\}$ and the delayed branching process $\{\hat Z_k^\pm: k \geq 1\}$ described above. As mentioned earlier, the value of this new coupling is that it holds for a number of steps in the graph exploration process that is equivalent to having discovered $n^{1-\delta}$ number of nodes for arbitrarily small $0 < \delta < 1$; moreover, the coupled branching process is independent of the bi-degree sequence and of the number of nodes.  Throughout the remainder of the paper, $\varepsilon > 0$ and $0 < \kappa
\leq 1$ are those from Assumption~\ref{A.Wasserstein}. 

\begin{thm} \label{thm:main_coupling}
	Suppose that $({\bf D}_n^-, {\bf D}_n^+)$ satisfies Assumption \ref{A.Wasserstein}. Then, for any $0 < 
	\delta < 1$, any $0 < \gamma < \min\{ \delta\kappa, \varepsilon\}$,  there exist finite constants $K, a > 0$ such that
	for all $1 \leq k \le (1-\delta) \log_\mu  n$, 
	\[
		\mathbb{P}_n\left( \bigcap_{m=1}^{k} \left\{ \left| \hat A^\pm_{m} \cap (A^\pm_m)^c \right| 
		\leq \hat Z_m^\pm  n^{-\gamma}, \,   \left| A^\pm_{m} \cap (\hat A^\pm_m)^c \right| \leq  \hat 
		Z_m^\pm  n^{-\gamma}  \right\} \right) \geq 1 -  K n^{-a}.
	\]
\end{thm}

As an immediate corollary, relation \eqref{eq:MainInequality} gives:

\begin{cor} \label{C.MainCoupling}
	Suppose that $({\bf D}_n^-, {\bf D}_n^+)$ satisfies Assumption \ref{A.Wasserstein}. Then, for any 
	$0 < \delta < 1$, any $0 < \gamma < \min\{ \delta\kappa, \varepsilon\}$, there exist finite constants $K, a > 0$ such that for all $1 \leq k \le (1-\delta) \log_\mu n$, 
	\[
		\mathbb{P}_n \left( \bigcap_{m=1}^{k} \left\{ \hat Z_m^\pm \left(1 -  n^{-\gamma} \right) 
		\leq Z_m^\pm \leq \hat Z_m^\pm \left(1 + n^{-\gamma} \right) \right\} \right) 
		\geq 1 - K n^{-\alpha}.
	\]
\end{cor}


\section{Distances in the directed configuration model}\label{S.Distances}

Having described the graph exploration process in the previous section, we are now ready to derive an expression for the hopcount between two randomly chosen nodes in a directed graph of size $n$ generated via the DCM. The main result of this section is Theorem~\ref{T.MainTheoremBody}, which expresses the tail 
distribution of the hopcount in terms of limiting random variables related to the branching 
processes $\{ {\hat Z}^+_k: k \geq 1\}$ and $\{ {\hat Z}^-_k: k \geq 1\}$ introduced in the previous section. Although we will 
include some preliminary calculations here, we refer the reader to Section
\ref{SS.ProofsDistances} for all other proofs.  

As described in Section~\ref{S.Coupling}, we will compute the hopcount of a graph by 
selecting two nodes at random, say $1$ and $2$, and then start two independent breadth-first exploration 
processes. One will follow the outbound edges of node $1$ while the other will use the inbound 
edges of node $2$. At each step we explore one generation of the out-component of node 1 and the 
corresponding generation of the in-component of node 2, starting with node 1. 

In terms of the two nodes, $\{ Z^+_k: k \geq 1\}$ will denote the number of outbound stubs discovered during the $k$th step of exploration of the out-component of node 1, while $\{ Z^-_k: k \geq 1\}$ will denote the number of inbound stubs discovered during the $k$th step of the exploration of the in-component of node 2. An expression for the  distribution of the hopcount is then obtained by computing the probability that there are no nodes in common given the current number of stubs explored so far in each of the two processes. We point out that the hopcount may be in fact infinite, which happens when node 2 is not in the out-component of node 1.


The first step in the analysis is a recursive relation for $\Probn{H_n > k}$. For this we denote by 
$\mathscr{F}^{l, m} = \sigma(\Zin_i, \Zout_j : 0 \le i \le l, 0 \le j \le m)$ the sigma algebra 
generated by the $\Zin_i$ and $\Zout_j$ of the first $l$ and $m$ generations, respectively. The 
next result follows from the analysis done in~\cite{Hofstad2005} Lemma 4.1, which can be adapted to
our case in a straight forward fashion.
\begin{equation}
	\Probn{H_n > k} = \Expn{\prod_{i = 2}^{k + 1} \CProbn{H_n > i - 1}{H_n > i - 2, \mathscr{F}^{
	\ceil{i/2}, \floor{i/2}}}} \quad \text{for all } k \ge 1.
\label{eq:hopcount_conditional}
\end{equation}
The ceiling and floor functions are here because we iteratively advance the exploration process alternating between nodes $1$ and $2$, starting with $1$.

Let $p(A, B, L)$ denote the probability that none of the outbound stubs from a set of size $A$ 
connect to one of the inbound stubs from a set of size $B$, given that there are $L$ 
outbound/inbound stubs in total. Since we can only select $A$ inbound stubs outside of the set
of size $B$ if $A + B \le L$ and the probability of selecting the first such stub is $1 - B/L$, we 
get
$$p(A, B, L) = 1(A + B \le L) \left(1 - \frac{B}{L}\right)p(A - 1, B, L - 1).$$
Continuing the recursion yields,
\begin{equation}
	p(A, B, L) = 1\left(A + B\le L\right) \prod_{s = 0}^{A - 1}\left(1 - \frac{B}{
	L - s}\right).
\label{eq:disconnected_prob}
\end{equation}

Next, observe that $H_n > 1$ holds if and only if none of the $\Zout_1$ outgoing edges points 
towards node $2$. From the definition of the model this occurs if and only if none of the $\Zout_1$ 
outbound stubs have been paired with one of the $\Zin_1$ inbound stubs. Hence,
\[
	\CProbn{H_n > 1}{\mathscr{F}^{1, 1}} = p(\Zout_1, \Zin_1, L_n) 
	= \Ind\left( \Zout_1 + \Zin_1 \le L_n\right) \prod_{s = 0}^{\Zout_1 - 1} 
	\left(1 - \frac{\Zin_1}{L_n - s}\right).
\]
Similarly, we have
\[
	\CProbn{H_n > 2}{H_n > 1, \mathscr{F}^{2,1}} 
	= \Ind\left(\Zout_2 + \Zin_1 \le L_n - \Zout_1\right) \prod_{s = 0}^{\Zout_2 - 1} 
	\left(1 - \frac{\Zin_1}{L_n - \Zout_1 - s}\right).
\]
In order to write the full formula we first define $\{\mathscr{S}_k\}_{k \ge 0}$ as follows:
\begin{equation}
	\mathscr{S}_0 = 0, \quad \mathscr{S}_1 = \Zout_1, \quad \mathscr{S}_k = \sum_{j = 1}^{\ceil{k/2}} \Zout_j 
	+ \sum_{j = 1}^{\floor{k/2}} \Zin_j	\quad \text{for } k \ge 2.
	\label{eq:Bk}
\end{equation}
We then obtain, for $i \ge 2$,
\[
	\CProbn{H_n > i - 1}{H_n > i - 2, \mathscr{F}^{\ceil{i/2}, \floor{i/2}}}
	= 1(\mathscr{S}_{i} \le L_n)\prod_{s = 0}^{\Zout_{\ceil{i/2}} - 1} 
	\left(1 - \frac{\Zin_{\floor{i/2}}}{L_n - \mathscr{S}_{i - 2} - s}\right).
\]
Substituting this expression into~\eqref{eq:hopcount_conditional}  yields
\begin{align*}
	\mathbb{P}_n(H_n > k) &= \Expn{\Ind(\mathscr{S}_{k + 1} \le L_n) \prod_{i = 2}^{k + 1}
		\prod_{s = 0}^{\Zout_{\ceil{i/2}} - 1}
		\left(1 - \frac{\Zin_{\floor{i/2}}}{L_n - \mathscr{S}_{i - 2} - s}\right)}
		 \numberthis \label{eq:hopcount_main}
\end{align*}

The first result for the hopcount uses equation~\eqref{eq:hopcount_main} combined with Corollary
\ref{C.MainCoupling} to obtain an expression in terms of the branching processes 
$\{\hat Z_k^+: k \geq 1\}$ and $\{\hat Z_k^-: k \geq 1\}$. We use the notation $g(x) = O(f(x))$ as $x \to \infty$ if $\limsup_{x \to \infty} g(x)/f(x) < \infty$. 

\begin{prop} \label{P.HopCount}
Suppose that $({\bf D}_n^-, {\bf D}_n^+)$ satisfies Assumption \ref{A.Wasserstein}.  Then, for any 
$0 < \delta < 1$ and for any $0 \leq k \leq 2(1-\delta) \log_\mu n$, there exists a constant 
$a > 0$ such that 
\[
	\left| \mathbb{P}_n( H_n > k) - E\left[ \exp\left\{ -\frac{1}{\nu n} \sum_{i=2}^{k+1} 
	\hat Z_{\lceil i/2 \rceil}^{-} \hat Z_{	\lfloor i/2 \rfloor}^{+} \right\}  
	\right] \right| = O\left(n^{-a} \right), \qquad n \to \infty,
\]
where $\{ \hat Z_i^{+}: i \geq 1 \}$ and $\{ \hat Z_i^-: i \geq 1\}$ are independent delayed 
branching processes having offspring distributions $(g^+, f^+)$ and $(g^-, f^-)$, respectively.     
\end{prop}

The next result shows a simplified expression for the limit in Proposition \ref{P.HopCount} in 
terms of the martingale limits $W^+$ and $W^-$. This result is independent of the coupling, and 
follows from the properties of the (delayed) branching processes $\{\hat Z_k^+: k \geq 1\}$ and 
$\{\hat Z_k^-: k \geq 1\}$. We state it here since it plays an important role in establishing
both Theorem~\ref{thm:main_result} and Proposition~\ref{prop:finite_paths}.

\begin{prop} \label{P.BPapprox}
Suppose $\{ \hat Z^+_i: i \geq 1\}$ and $\{ \hat Z^-_i: i \geq 1\}$ are independent delayed 
branching processes having offspring distributions $(g^+, f^+)$ and $(g^-, f^-)$, respectively. 
Suppose that $f^+, f^-$ have finite moments of order $1+\kappa \in (1, 2]$ with common mean 
$\mu > 1$, and $g^+, g^-$ have common mean $\nu$. Then, there exists $b > 0$ such that 
\[
	\left| E\left[ \exp\left\{-\frac{1}{\nu n} \sum_{i=2}^{k+1} \tbtZout_{\ceil{i/2}} 
		\tbtZin_{\floor{i/2}} \right\} - \exp\left\{-\frac{\nu \mu^{k}}{(\mu-1)n} W^+ W^- \right\} 
		\right] \right|= O\left( n^{-b} \right), \qquad n \to \infty,
\]
uniformly for all $k \in \mathbb{N}_+$, where $W^\pm = \lim_{k \to \infty} \hat Z^\pm_k/(\nu \mu^{k-1})$.
\end{prop}

Combining Propositions \ref{P.HopCount} and \ref{P.BPapprox}, we immediately obtain the following 
result.

\begin{thm} \label{T.MainTheoremBody}
Suppose $({\bf D}_n^-, {\bf D}_n^+)$ satisfies Assumption \ref{A.Wasserstein}.  Then, for any 
$0 < \delta < 1$ and for any $0 \leq k \leq 2(1-\delta) \log_\mu n$, there exists a constant 
$c > 0$ such that
\[
	\left| \mathbb{P}_n(H_n > k) - E\left[ \exp \left\{ - \frac{\nu \mu^{k}}{(\mu-1)n} W^- W^+ 
	\right\} \right] \right| = O\left( n^{-c} \right), \qquad n \to \infty,
\]
where $W^\pm = \lim_{k \to \infty} \hat Z^\pm_k/(\nu \mu^{k-1})$, with $W^+$ and $W^-$ independent 
of each other.
\end{thm}

As a corollary of Theorem~\ref{T.MainTheoremBody} we obtain the following result for the probability that there exists a directed path between two randomly chosen nodes, which implies Proposition~\ref{prop:finite_paths}.

\begin{cor} \label{C.FiniteHopcount}
Suppose $({\bf D}_n^-, {\bf D}_n^+)$ satisfies Assumption \ref{A.Wasserstein}. Then, there exists a 
constant $c> 0$ such that 
\[
	\left| \mathbb{P}_n(H_n < \infty) - s^+ s^- \right| = O\left( n^{-c} \right), \qquad n \to \infty,
\]
where $s^\pm = P(W^\pm > 0)$. 
\end{cor}

Noting that
$$\mathbb{P}_n(H_n > k) = \mathbb{P}_n(H_n > k | H_n < \infty) \mathbb{P}_n(H_n < \infty) + \mathbb{P}_n(H_n = \infty),$$
defining $\mathcal{B} = \{ W^+ W^- > 0\}$, and using Theorem~\ref{T.MainTheoremBody} and Corollary~\ref{C.FiniteHopcount} gives
\begin{align*}
\mathbb{P}_n(H_n > k | H_n < \infty) &= \frac{\mathbb{P}_n(H_n > k) - \mathbb{P}_n(H_n = \infty)}{\mathbb{P}_n(H_n < \infty) } \\
&= \frac{1}{P(\mathcal{B})} E\left[ \exp \left\{ - \frac{\nu \mu^{k}}{(\mu-1)n} W^- W^+ 
	\right\} \right]  - \frac{P(\mathcal{B}^c)}{P(\mathcal{B})} + O\left( n^{-c} \right) \\
	&= E\left[ \left. \exp \left\{ - \frac{\nu \mu^{k}}{(\mu-1)n} W^- W^+ 
	\right\} \right| W^+ W^- > 0 \right] + O\left( n^{-c} \right)
\end{align*}
as $n \to \infty$ and for the range of values of $k$ indicated in the theorems. Now define for $x \in \mathbb{R}$,
\[
	V_n(x) = 1 -  E\left[\left.\exp\left\{ - \frac{ \nu \mu^{\lfloor \log_\mu n \rfloor + \lfloor x \rfloor}}{(\mu-1) n} W^+ W^-\right\} 
	\right| W^+ W^- > 0\right].
\]
That $V_n(x)$ is a cumulative distribution function for each fixed $n$ follows from noting that it is non-decreasing with $\lim_{x \to -\infty} V_n(x) = 0$ and $\lim_{x \to \infty} V_n(x) = 1$. Letting $\mathcal{H}_n$ be a random variable having distribution $V_n$ gives Theorem \ref{thm:main_result}.

The remainder of the paper is devoted to the proofs of all the results presented in Sections \ref{S.Coupling} and~\ref{S.Distances}.  


\section{Proofs} \label{S.Appendix}

This section consists of four subsections. In Section \ref{SS.BPs} we prove some general results about delayed branching processes, including a bound for its minimum growth conditional on non-extinction.  Section~\ref{SS.CouplingProofs} contains the proof of Theorem \ref{thm:main_coupling}, our main coupling theorem. The proofs of our results for the hopcount, Proposition~\ref{P.HopCount}, Proposition~\ref{P.BPapprox}, Theorem~\ref{T.MainTheoremBody}, and Corollary~\ref{C.FiniteHopcount}, are given in Section~\ref{SS.ProofsDistances}. Finally, Section~\ref{SS.TechnicalLemmata} contains the proof of Theorem \ref{T.IIDAlgorithm}, which shows that the i.i.d.~algorithm given in Section \ref{SS.IIDAlgorithm} satisfies the main assumptions in the paper.

\subsection{Some results for delayed branching processes} \label{SS.BPs}

Our first result for a general delayed branching process is an expression for its extinction probability in terms of the probability of extinction of the corresponding non-delayed process, as well as for the distribution of its number of offspring  conditional on extinction. Since these results are independent of the coupling with the graph, we do not use the $\pm$ notation.

\begin{lemma} \label{L.Extinction}
Let $\{\mathcal{Z}_k: k \geq 0\}$ denote a (non-delayed) branching process having offspring 
distribution $f$ and extinction probability $q$ and let $\{\hat Z_k: k \geq 1\}$ be a delayed 
branching process having offspring distributions $(g, f)$. Suppose $q > 0$. Then, conditioned on 
extinction, $\{ \hat Z_k: k \geq 1\}$ is a delayed branching process with offspring distributions $(\tilde g, \tilde f)$ with
\[
	\tilde g(i) = \frac{ g(i) q^i}{\sum_{t=0}^\infty g(t) q^t} \qquad \text{and} \qquad 
	\tilde f(i) = f(i) q^{i-1}, \qquad i \geq 0.
\]
Moreover, $P\left( \hat Z_k = 0 \text{ for some $k \geq 1$} \right) = \sum_{t=0}^\infty g(t) q^t$. 
\end{lemma}

\begin{proof}
Let $\hat \chi_\emptyset$ have distribution $g$ and let $\{ \mathcal{Z}_{k-1,i}\}_{i \geq 1}$ be a sequence of i.i.d.~copies of $\mathcal{Z}_{k-1}$, independent of $\hat \chi_\emptyset$; set $\mu$ to be the mean of $f$. Computing the probability generating function of $\hat Z_k$ we obtain
\begin{align*}
P(W = 0) E\left[ \left. s^{\hat Z_k} \right| W = 0\right] 
&= E\left[ \left( E\left[ \left. s^{\mathcal{Z}_{k-1} } \right| \mathcal{W} = 0 \right] q \right) ^{\hat \chi_\emptyset}  \right] ,
\end{align*}
where $\mathcal{W}_i$ is the a.s. limit of the martingale $\{\mathcal{Z}_{k,i}/\mu^k: k \geq 0\}$ that has as root the $i$th individual in the first generation of $\{\hat Z_k: k \geq 1\}$.  Also, 
$$P(W = 0) = P( \hat \chi_\emptyset = 0) + P\left( \hat \chi_\emptyset \geq 1, \, \bigcap_{i=1}^{\hat \chi_\emptyset} \{ \mathcal{W}_i = 0 \} \right) = \sum_{j=0}^\infty g(j) q^j.$$
Hence,
$$E\left[ \left. s^{\hat Z_k} \right| W = 0\right] = \sum_{j=0}^\infty \left( E\left[ \left. s^{\mathcal{Z}_{k-1} } \right| \mathcal{W} = 0 \right] \right)^j \tilde g(j),$$ 
where 
$$\tilde g(j) = \frac{q^j g(j)}{\sum_{t=0}^\infty g(t) q^t}, \qquad j \geq 0.$$
Since conditionally on extinction $\{\mathcal{Z}_k : k \geq 0\}$ is a subcritical (non-delayed) branching process with offspring distribution $\tilde f(j) = f(j) q^{j-1}$, $j \geq 0$ (see, e.g., \cite{Athreya2012}, pp. 52), the result follows. 
\end{proof}

The second result we show is in some sense the counterpart of Doob's 
maximal martingale inequality, and it states that provided the limiting martingale is strictly positive, the 
branching process itself cannot grow too slowly. For this result and others in this section, we use the following 
version of Burkholder's inequality, which we state without proof. 
\begin{lemma} \label{L.Burkholder}
Let $\{X_i\}_{i\geq 1}$ be a sequence of i.i.d., mean zero random variables such that $E[|X_1|^{1+\kappa}] < \infty$ for some $0 < \kappa \leq 1$. Then,
$$P\left( \left| \sum_{i=1}^n X_i \right| > x \right) \leq \frac{1}{x^{1+\kappa}} E\left[ \left| \sum_{i=1}^n X_i \right|^{1+\kappa} \right] \leq Q_{1+\kappa} E[ |X_1|^{1+\kappa}]  \frac{n}{x^{1+\kappa}},$$
where $Q_{1+\kappa}$ is a constant that depends only on $\kappa$. 
\end{lemma}
%

\begin{lemma} \label{L.Martingale}
Suppose $\{\hat Z_k: k \geq 1\}$ is a delayed branching process with offspring distributions 
$(g, f)$, where
$f$ has finite $1+\kappa \in (1, 2]$ moment and mean $\mu > 1$, and $g$ has finite mean $\nu> 0$. 
Let $W = \lim_{k \to \infty} \hat Z_k/(\nu \mu^{k-1})$. Then, for any $1 < u < \mu$, there exists a constant $Q_1 < \infty$ such that for any $k \geq 1$, 
\[
	P\left( \inf_{r \geq k} \frac{\hat Z_r}{u^{r}} < 1, \,  W > 0 \right) \leq Q_1 \left( 
	u^{-\kappa k} + (u/\mu)^{\alpha k} 1(q > 0) \right),
\]
where $q$ is the extinction probability of a branching process having offspring distribution $f$,  $\lambda = \sum_{i=1}^\infty f(i) i q^{i-1}$, and $\alpha = -\log \lambda/\log\mu > 0$ if $q > 0$. 
\end{lemma}

\begin{proof}
We start by defining for $r \geq k$ the event $D_r = \{ \min_{k \leq j \leq r} \hat Z_j/u^j \geq 
1\}$ and letting $a_r = P(W > 0, \, (D_r)^c)$. Let $\{\hat \chi, \hat \chi_i \}$ be a sequence of 
i.i.d. random variables having distribution $f$ and use the technical Lemma \ref{L.Burkholder}, applied 
conditionally on $\hat Z_{r-1}$, to obtain
\begin{align*}
a_r &\leq P\left( D_{r-1}, \, \hat Z_r \leq u^{r}  \right) + a_{r-1} \\
&\leq P\left( \hat Z_{r-1} \geq u^{r-1}, \, \sum_{i=1}^{\hat Z_{r-1}} \hat \chi_i \leq u \hat Z_{r-1}    \right) + a_{r-1} \\
&\leq P\left( \hat Z_{r-1} \geq u^{r-1}, \, \sum_{i=1}^{\hat Z_{r-1}} (\mu - \hat \chi_i) \geq (\mu-u) \hat Z_{r-1}   \right) + a_{r-1} \\
&\leq E\left[  1( \hat Z_{r-1} \geq u^{r-1}) \frac{Q_{1+\kappa} E[|\hat \chi - \mu|^{1+\kappa}]}{(\mu-u)^{1+\kappa} (\hat Z_{r-1})^{\kappa}}   \right] + a_{r-1} \\
&\leq Q u^{-\kappa(r-1)}  + a_{r-1},
\end{align*}
where $Q = Q_{1+\kappa} E[|\hat \chi - \mu|^{1+\kappa}] / (\mu-u)^{1+\kappa}$ and $E[ (\hat \chi)^{1+\kappa}] < \infty$ by Remark \ref{R.LimitMoments}. It follows from iterating the inequality derived above that
\begin{align*}
a_r &\leq Q \sum_{j= k}^{r-1} \frac{1}{u^{\kappa j}} + a_{k} \leq \frac{Q}{(u^\kappa-1) u^{\kappa (k-1)}} + P( W > 0, \, \hat Z_{k} < u^{k} )
\end{align*}
for all $r \geq k$. It remains to bound the last probability. 

Let $\{\mathcal{Z}_k: k \geq 0\}$ be a (non-delayed) branching process 
 having offspring distribution $f$, and let $\mathcal{W} = \lim_{k \to \infty} \mathcal{Z}_k/\mu^k$. It is well known (see \cite{Athreya2012}, pp. 52), that conditional on non-extinction,  $\mathcal{W}$ has an absolutely continuous distribution on $(0,\infty)$. Note also that for any $m \geq 1$ we have
$$W_{m+k} = \frac{\hat Z_{m+k}}{\nu \mu^{m+k-1}} = \frac{1}{\nu \mu^{m+k-1}} \sum_{{\bf i} \in \hat A_{k}} \mathcal{Z}_{m,{\bf i}},$$
where the $\{ \mathcal{Z}_{m,{\bf i}} \}$ are i.i.d.~copies of $\mathcal{Z}_m$ and $\hat A_k$ is the set of individuals in the $k$th generation of $\{\hat Z_k: k \geq 1\}$, and therefore, for any $k \geq 1$,
$$W_{m+k} - W_k = \frac{1}{\nu \mu^{k-1}} \sum_{{\bf i} \in \hat A_{k}} \left( \frac{\mathcal{Z}_{m,{\bf i}}}{\mu^m} -  1 \right).$$
Now define $\mathcal{W}_{\bf i} = \lim_{m \to \infty} \mathcal{Z}_{m,{\bf i}} / \mu^m$ to obtain that
\begin{equation} \label{eq:ConvergenceRate}
W - W_k = \frac{1}{\nu \mu^{k-1}} \sum_{{\bf i} \in \hat A_{k}} \left( \mathcal{W}_{\bf i} -  1 \right),
\end{equation}
where the $\{ \mathcal{W}_{\bf i}\}$ are i.i.d.~copies of $\mathcal{W}$, independent of the history of the tree up to generation $k$. It follows that for $x_k = 2 u^k/(\nu \mu^{k-1})$, 
\begin{align}
P\left( \hat Z_k < u^k, \, W > 0 \right) &\leq P\left(  \hat Z_k < u^k, \,W \geq x_k \right) + P\left( 0 <  W < x_k \right) \notag \\
&\leq P\left(\hat Z_k  < u^k, \, W - W_k \geq x_k - u^k/(\nu \mu^{k-1})  \right) + P\left( 0 <  W < x_k \right) \notag  \\
&= E\left[ 1(\hat Z_k < u^k) P\left( \left. \sum_{{\bf i} \in \hat A_k} (\mathcal{W}_{\bf i} -1) \geq u^k \right| \hat Z_k \right) \right] + P\left( 0 <  W < x_k \right) \notag .
\end{align}
Now note that $E[ \hat \chi^{1+\kappa} ] < \infty$ implies that $E[ \mathcal{W}^{1+\kappa}] < \infty$. Then, by Lemma \ref{L.Burkholder}, applied conditionally on $\hat Z_k$, we obtain
\begin{align*}
P\left( \left. \sum_{{\bf i} \in \hat A_k} (\mathcal{W}_{\bf i} -1) \geq u^k \right| \hat Z_k \right) &\leq   Q_{1+\kappa} E[|\mathcal{W} -1|^{1+\kappa}] \cdot \frac{ \hat Z_k}{u^{(1+\kappa)k}}. 
\end{align*}
Hence,
$$P\left( \hat Z_k < u^k, \, W > 0 \right) \leq \frac{Q_{1+\kappa} E[|\mathcal{W} -1|^{1+\kappa}]}{u^{\kappa k}} + P(0 < W < x_k). $$

Finally, to bound $P(0 < W < x_k)$, note that  $W$ admits the representation
$$W = \frac{1}{\nu} \sum_{i=1}^{\hat \chi_\emptyset} \mathcal{W}_i,$$
where the $\{\mathcal{W}_i\}$ are i.i.d.~copies of $\mathcal{W}$, independent of $\hat \chi_\emptyset$, with $\hat \chi_{\emptyset}$ distributed according to $g$. Note that  $W > 0$ implies that at least one of the $\mathcal{W}_i$ is strictly positive. Let $N(t)$ be the number of non-zero random variables among $\{ \mathcal{W}_1, \dots, \mathcal{W}_t\}$. It follows that if we let $\{V_i\}$ be i.i.d.~random variables having the same distribution as $\mathcal{W}$ given $\mathcal{W} > 0$, then 
\begin{align*}
P\left( 0 < W < x_k \right) &= P\left( \frac{1}{\nu} \sum_{i=1}^{N(\hat \chi_\emptyset)} V_i < x_k, \, N(\hat \chi_{\emptyset}) \geq 1 \right) \leq P\left( V_1 < \nu x_k \right) .
\end{align*}
Hence, if $w(t)$ denotes the density of $\mathcal{W}$ conditional on non-extinction, we have that  
$$P(V_1 < \nu x_k) = P\left( \left. \mathcal{W} < \nu x_k \right| \mathcal{W} > 0 \right) = \int_0^{\nu x_k} w(t) \, dt.$$
By Theorem 1 in \cite{SergeDubuc1971} (see also Theorem 4 in \cite{Biggins1993}), we have that if $\lambda = \sum_{i=1}^\infty f(i) i q^{i-1} > 0$, which under the assumptions of the lemma occurs whenever $q > 0$, then there exists a constant $C_0 < \infty$ such that
$$\int_0^{\nu x_k} w(t) \, dt \leq C_0 (\nu x_k)^{\alpha}$$
for $\alpha = -\log \lambda/\log \mu$; whereas if $f(0) + f(1) = 0$, then Theorem 3 in \cite{Biggins1993} gives that for every $a > 0$ there exists a constant $C_a < \infty$ such that
$$\int_0^{\nu x_k} w(t) \, dt \leq C_a (\nu x_k)^a.$$
We conclude that for $a^* = \kappa \log u/\log(\mu/u)$, 
\begin{align*}
P\left( \min_{r \geq k} \frac{\hat Z_r}{u^r} < 1, \, W > 0 \right) &\leq \frac{Q}{(u^\kappa -1) u^{\kappa(k-1)}} + \frac{Q_{1+\kappa} E[|\mathcal{W} -1|^{1+\kappa}]}{u^{\kappa k}} \\
&\hspace{5mm} + C_0 (\nu x_k)^{\alpha} 1(q > 0) + C_{a^*} (\nu x_k)^{a^*}  \\
&\leq Q_1 \left( u^{-\kappa k} + (u/\mu)^{\alpha k} 1(q > 0) \right) .
\end{align*}
\end{proof}


\subsection{Coupling with a branching process}  \label{SS.CouplingProofs}

In this section we prove Theorem \ref{thm:main_coupling}. As mentioned in Section \ref{S.Coupling}, the coupling we constructed is based on bounding the Kantorovich-Rubinstein distance between the distributions $H_{\bf i}^\pm$ and $F^\pm$, and the main difficulty lies in the fact that this distance grows as the number of explored stubs in the graph grows. The proof of the main theorem is based on four technical results, Lemmas~\ref{L.KRdistance}, \ref{L.Expectations}, \ref{L.BoundE} and Proposition~\ref{P.OnExtinction}, which we state and prove below. 

Throughout this section let
$$Y_k^\pm = \sum_{r=1}^k Z_r^\pm, \quad k \geq 1; \qquad Y_0^\pm = 0.$$

The first of the technical lemmas gives us an upper bound for the Kantorovich-Rubinstein distance conditionally on the history of the graph exploration process and its coupled tree up to the moment that stub ${\bf i}$ is about to be traversed. 

\begin{lemma} \label{L.KRdistance}
Let $\mathcal{G}_{\bf i}$ denote the sigma-algebra generated by the bi-degree sequence $({\bf D}_n^-, {\bf D}_n^+)$ and the graph exploration process up to the time that outbound (inbound) stub {\bf i} is about to be traversed. Then, provided $({\bf D}_n^-, {\bf D}_n^+)$ satisfies Assumption \ref{A.Wasserstein}, for all $n$ sufficiently large, and for $T_{\bf i}^\pm \leq (\nu/2) n$,  we have
$$\mathbb{E}_n \left[ \left. \left| \hat \chi_{\bf i}^\pm - \chi_{\bf i}^\pm \right| \right| \mathcal{G}_{\bf i} \right] \leq \mathcal{E}(T_{\bf i}^\pm),$$
where
\begin{equation} \label{eq:ErrorFunction}
\mathcal{E}(t) = \frac{4}{\nu n} \sum_{r=1}^n (1 - \mathcal{I}_r(t))D_r^+ D_r^- + \frac{4\mu t}{\nu n} + 3n^{-\varepsilon}.
\end{equation}
\end{lemma}

\begin{proof}
We first point out that for ${\bf i} = \emptyset$ the result holds trivially by Assumption \ref{A.Wasserstein}, since
$$\mathbb{E}_n \left[ \left| \hat \chi_{\emptyset}^\pm - \chi_\emptyset^\pm \right| \right] = d_1(G_n^\pm, G^\pm) \leq n^{-\varepsilon}.$$

For ${\bf i} \neq \emptyset$ we have
$$\mathbb{E}_n \left[ \left. \left| \hat \chi_{\bf i}^\pm - \chi_{\bf i}^\pm \right| \right| \mathcal{G}_{\bf i} \right] = d_1(H_{\bf i}^\pm, F^\pm) \leq d_1(H_{\bf i}^\pm, F_n^\pm) + d_1(F_n^\pm, F^\pm).$$
Since by Assumption \ref{A.Wasserstein} we have that the second distance is smaller or equal than $n^{-\varepsilon}$, we only need to analyze the first one, which we do separately for the $+$ and $-$ cases. To this end write
\begin{align*}
d_1(H_{\bf i}^+, F_n^+) &= \sum_{k=0}^\infty \left| \sum_{j=0}^k (h_{\bf i}^+ (j) - f_n^+(j)) \right| = \sum_{k=0}^\infty \left| \sum_{j=k+1}^\infty (f_n^+(j) - h_{\bf i}^+ (j) ) \right| \\
&= \sum_{k=0}^\infty \left|   \sum_{j=k+1}^\infty \sum_{r=1}^n 1(D_r^+ = j) D_r^-  \left(  \frac{\mathcal{I}_r(T_{\bf i}^+)}{L_n - T_{\bf i}^+} -  \frac{1}{L_n} \right) \right| \\
&\leq \sum_{k=0}^\infty  \sum_{r=1}^n  \left|  \frac{\mathcal{I}_r(T_{\bf i}^+)}{L_n - T_{\bf i}^+} -  \frac{1}{L_n} \right|   D_r^- 1(D_r^+ > k)   \\
&= \sum_{r=1}^n  \left|  \frac{(\mathcal{I}_r(T_{\bf i}^+) - 1) L_n + T_{\bf i}^+}{(L_n - T_{\bf i}^+) L_n} \right|   D_r^+ D_r^- \\
&\leq \frac{1}{L_n - T_{\bf i}^+} \sum_{r=1}^n (1 - \mathcal{I}_r(T_{\bf i}^+)) D_r^+ D_r^- + \frac{T_{\bf i}^+}{L_n - T_{\bf i}^+} \cdot \mu_n ,
\end{align*}
where $\mu_n = L_n^{-1} \sum_{r=1}^n D_r^+ D_r^-$ is the common mean of $F_n^+$ and $F_n^-$. Symmetrically,
$$d_1(H_{\bf i}^-, F_n^-) \leq  \frac{1}{L_n - T_{\bf i}^-} \sum_{r=1}^n (1 - \mathcal{I}_r(T_{\bf i}^-)) D_r^+ D_r^- + \frac{T_{\bf i}^-}{L_n - T_{\bf i}^-} \cdot \mu_n.$$

Now note that 
$$\mu_n \leq \mu + d_1(F_n^\pm, F^\pm),$$
and if $\nu_n$ denotes the common mean of $G_n^+$ and $G_n^-$, then
$$\frac{L_n}{n} = \nu_n \geq \nu - d_1(G_n^\pm, G^\pm),$$
which in turn implies that
$$(L_n - T_{\bf i}^\pm)^{-1} \leq (\nu n - T_{\bf i}^\pm - n d_1(G_n^\pm, G^\pm))^{-1}.$$

We conclude that, under Assumption \ref{A.Wasserstein} and for $T_{\bf i}^\pm \leq (\nu/2)n$, 
\begin{align*}
d_1(H_{\bf i}^\pm, F^\pm) &\leq \frac{1}{(\nu/2) n - n^{1-\varepsilon} } \sum_{r=1}^n (1 - \mathcal{I}_r(T_{\bf i}^\pm)) D_r^+ D_r^-  + \frac{T_{\bf i}^\pm}{(\nu/2) n - n^{1-\varepsilon}} \cdot (\mu + n^{-\varepsilon}) + n^{-\varepsilon} \\
&\leq \frac{4}{\nu n } \sum_{r=1}^n (1 - \mathcal{I}_r(T_{\bf i}^\pm)) D_r^+ D_r^- + \frac{4\mu T_{\bf i}^\pm}{\nu n} + 3 n^{-\varepsilon}
\end{align*}
for all $n \geq (4/\nu)^{1/\varepsilon}$. 
\end{proof}

The second preliminary result provides an estimate for the expected value of the bound obtained in the previous lemma on the set where $\{ \hat Z_k^\pm: k \geq 0\}$ behaves typically, i.e., without exhibiting large deviations from its mean.

\begin{lemma} \label{L.BoundE}
Define $\mathcal{E}(t)$ according to \eqref{eq:ErrorFunction}, and for any fixed $0 < \eta < 1$ and all $m \geq 1$ define the event
\begin{equation} \label{eq:DefEk}
E_m = \bigcap_{r=1}^m \left\{ \hat Z_r^\pm /\mu^r \leq n^{\eta} \right\}.
\end{equation}
Then, provided $({\bf D}_n^-, {\bf D}_n^+)$ satisfies Assumption~\ref{A.Wasserstein}, there exists a constant $Q_2 < \infty$ such that for any $0 \leq t \leq \nu n/2$ and any $k \geq 1$,
$$\mathbb{E}_n[ \mathcal{E}(t) ] \leq Q_2 \left( \frac{ t^\kappa }{n^\kappa} + n^{-\varepsilon} \right) \quad \text{and} \quad \mathbb{E}_n \left[ 1(E_k) \hat Z_k^\pm \mathcal{E}(t) \right] \leq Q_2 \mu^k \left( \frac{t^\kappa}{n^{\kappa(1-\eta)}} + n^{-\varepsilon} \right)$$
where $0 < \varepsilon < 1$ and $0 < \kappa \leq 1$ are those from Assumption \ref{A.Wasserstein}. 
\end{lemma} 

\begin{proof}
We start by proving the bound for $\mathbb{E}_n \left[ \mathcal{E}(t) \right]$. Let $X_r$ denote either $D_r^+$ or $D_r^-$ depending on whether we are exploring outbound stubs or inbound stubs, respectively. Recall that $\mathcal{I}_r(t)$ is the indicator of node $r$ being {\em inactive} at time $t$ in the graph exploration process. Next, note that $\mathcal{I}_r(0) = 1$, $\mathbb{E}_n[\mathcal{I}_r(1)] = 1- 1/n$, and for any $2 \leq t < L_n$,
\begin{align*}
\mathbb{E}_n\left[  \mathcal{I}_r(t) \right] &= \left(1 - \frac{1}{n} \right) \prod_{s=1}^{t-1} \left(1 - \frac{X_r}{L_n - s} \right) \geq \left(1 - \frac{1}{n} \right) \left( 1 - \frac{X_r}{L_n-t} \right)^{t-1} ,
\end{align*}
from where it follows that $\mathbb{E}_n[1- \mathcal{I}_r(0)] = 0$, $\mathbb{E}_n[1- \mathcal{I}_r(1)] = 1/n$, and for $2 \leq t < L_n$, 
\begin{align*}
\mathbb{E}_n\left[  1 - \mathcal{I}_r(t) \right] &\leq 1 - \left(1 - \frac{1}{n} \right) \left( 1- \frac{X_r}{L_n - t} \right)^{t-1} \\
&= \frac{1}{n} +  \left(1 - \frac{1}{n} \right) \frac{X_r}{L_n - t} \sum_{s=0}^{t-2} \left( 1- \frac{X_r}{L_n-t } \right)^s \\
&\leq \frac{1}{n} + \frac{(t-1) X_r}{L_n-t}.
\end{align*}

Now let $0 < \kappa \leq 1$ be the one from Assumption \ref{A.Wasserstein} and note that since $\mathbb{E}_n\left[  1 - \mathcal{I}_r(t) \right] \leq 1$, then,
\begin{equation*}
\mathbb{E}_n\left[  1 - \mathcal{I}_r(t) \right] \leq \left( \mathbb{E}_n\left[  1 - \mathcal{I}_r(t) \right] \right)^\kappa \leq \left( \frac{1}{n} \right)^\kappa + \left( \frac{ (t-1) X_r}{L_n-t } \right)^\kappa,
\end{equation*}
where we used the inequality $\left( \sum_i y_i \right)^\beta \leq \sum_{i} y_i^\beta$ for $y_i \geq 0$ and $0 < \beta \leq 1$.
 It follows that under Assumption \ref{A.Wasserstein} and for any $2 \leq t < \nu n - n^{1-\varepsilon}$, 
 \begin{align}
\frac{1}{n} \sum_{r=1}^n \left( \mathbb{E}_n\left[  1 - \mathcal{I}_r(t) \right] \right)^\kappa D_r^+ D_r^- &\leq \frac{1}{n^{1+\kappa}} \sum_{r=1}^n D_r^+ D_r^- + \frac{(t-1)^\kappa}{n (L_n-t)^\kappa} \sum_{r=1}^n ((D_r^+)^\kappa + (D_r^-)^\kappa) D_r^+ D_r^- \notag \\
 &\leq \frac{\nu_n \mu_n}{n^\kappa} + \frac{K_\kappa t^\kappa}{(L_n - t)^\kappa} \notag \\
 &\leq \frac{(\nu + n^{-\varepsilon}) (\mu + n^{-\varepsilon})}{n^\kappa} + \frac{K_\kappa t^\kappa}{(\nu n - n^{1-\varepsilon} - t)^\kappa}. \label{eq:HardBound}
 \end{align}
 
 It follows that $\mathbb{E}_n[\mathcal{E}(0)] \leq 3n^{-\varepsilon}$, $\mathbb{E}_n[ \mathcal{E}(1)] \leq 4\mu_n/(\nu n) + 3n^{-\varepsilon}$, and for any $2 \leq t \leq \nu n/2$ we have
 \begin{align*}
 \mathbb{E}_n[\mathcal{E}(t)] &\leq  \frac{4}{\nu n} \sum_{r=1}^n \mathbb{E}_n\left[  1 - \mathcal{I}_r(t) \right] D_r^+ D_r^- + \frac{4\mu t}{\nu n} + 3 n^{-\varepsilon} \\
 &\leq \frac{4\mu}{n^\kappa}  (1+O(n^{-\varepsilon}))  + \frac{K_\kappa t^\kappa}{(\nu n/2)^\kappa} (1+ O(n^{-\varepsilon})) + \frac{4\mu t}{\nu n} + 3 n^{-\varepsilon} \\
 &\leq Q_0 \left( \frac{t^\kappa}{n^\kappa} + n^{-\varepsilon} \right),
 \end{align*}
 for some constant $Q_0 < \infty$. 

Next, to compute a bound for $\mathbb{E}_n \left[ 1(E_k) \hat Z_k^\pm \mathcal{E}(t) \right]$ let $q = 1/\kappa$ and $p = q/(q-1)$, with $p = \infty$ if $q = 1$, and use H\"older's inequality to obtain, for $1 \leq t \leq \nu n/2$, 
\begin{align*}
\mathbb{E}_n \left[ 1( E_k) \hat Z_k^\pm \mathcal{E}(t) \right] &= \frac{4}{\nu n} \sum_{r=1}^n \mathbb{E}_n \left[ 1(E_k) \hat Z_k^\pm (1- \mathcal{I}_r(t)) \right] D_r^+ D_r^- \\
&\hspace{5mm} + \frac{4\mu t}{\nu n}  E \left[ 1(E_k) \hat Z_k^\pm \right] + 3 n^{-\varepsilon} E \left[ 1(E_k) \hat Z_k^\pm \right] \\
&\leq \frac{4 }{\nu n} \sum_{r=1}^n  \left(  E \left[ 1(E_k) \left( \hat Z_k^\pm \right)^p \right]  \right)^{1/p} \left(  \mathbb{E}_n \left[  1- \mathcal{I}_r(t) \right]  \right)^{1/q}  D_r^+ D_r^-\\
&\hspace{5mm}  + \frac{4\mu t}{\nu n } E \left[ \hat Z_k^\pm \right]  + 3 n^{-\varepsilon} E \left[ \hat Z_k^\pm \right] .
\end{align*}
Now note that 
$$\left(  E \left[ 1(E_k) \left( \hat Z_k^\pm \right)^p \right]  \right)^{1/p} \leq \left(  (\mu^{k} n^{\eta})^{p-1} E \left[  \hat Z_k^\pm  \right]  \right)^{1/p} =  \mu^k n^\eta   \left( \frac{\nu}{\mu n^\eta} \right)^{1/p}  = \left(\frac{\mu}{\nu}\right)^{\kappa} E \left[  \hat Z_k^\pm  \right] n^{\kappa\eta} .$$
Combining this inequality with \eqref{eq:HardBound} gives, for $1 \leq t \leq \nu n/2$,
\begin{align*}
\mathbb{E}_n \left[ 1( E_k) \hat Z_k^\pm \mathcal{E}(t) \right] &\leq E \left[ \hat Z_k^\pm \right]  \left( \frac{\mu^\kappa n^{\kappa\eta}}{\nu^\kappa} \cdot \frac{4 }{\nu n} \sum_{r=1}^n (\mathbb{E}_n[1- \mathcal{I}_r(t)])^\kappa D_r^+ D_r^- + \frac{4\mu t}{\nu n} + 3 n^{-\varepsilon} \right) \\
&\leq E \left[ \hat Z_k^\pm \right] \left(  \frac{\mu^\kappa n^{\kappa\eta}}{\nu^\kappa}  \left\{ \frac{4\mu}{n^\kappa}  (1+O(n^{-\varepsilon}))  + \frac{K_\kappa t^\kappa}{(\nu n/2)^\kappa} (1+ O(n^{-\varepsilon})) \right\} + \frac{4\mu t}{\nu n} + 3 n^{-\varepsilon}   \right) \\
&\leq Q_0' \mu^k \left( \frac{t^\kappa}{n^{\kappa(1-\eta)}} + n^{-\varepsilon } \right) ,
\end{align*}
for some constant $Q_0' < \infty$. Noting that $\mathbb{E}_n\left[ 1(E_k) \hat Z_k^\pm \mathcal{E}(0) \right] \leq E[\hat Z_k^\pm]  3 n^{-\varepsilon} = O( \mu^k n^{-\varepsilon}) $  completes the proof. 
\end{proof}

The third technical lemma provides an estimate for the expected number of stubs that are discovered during step $k+1$ of the graph exploration process, on the set where the coupling holds uniformly well up to step $k$. This bound is the key component that will enable the induction step in the proof of Theorem~\ref{thm:main_coupling}.

\begin{lemma} \label{L.Expectations}
Let $E_k$ be defined according to \eqref{eq:DefEk}. Fix $0 < \delta < 1$, $0 < \gamma < \min\{ \delta \kappa, \varepsilon\}$, and define
\begin{equation} \label{eq:DefCk}
C_m = \bigcap_{r=1}^m \left\{ \left| \hat A_r^\pm \cap (A_r^\pm)^c \right| \leq \hat Z_r^\pm n^{-\gamma}, \, \left| A_r^\pm \cap (\hat A_r^\pm)^c \right| \leq \hat Z_r^\pm n^{-\gamma} \right\}.
\end{equation}
Then, provided $({\bf D}_n^-, {\bf D}_n^+)$ satisfies Assumption \ref{A.Wasserstein}, there exists a constant $Q_3 < \infty$ such that for all $1 \leq k \leq (1-\delta) \log n/\log \mu$, 
$$\mathbb{E}_n \left[ 1(C_{k} \cap E_{k})  \left( \left|\hat A_{k+1}^\pm \cap (A_{k+1}^\pm)^c \right| + \left| A_{k+1}^\pm \cap (\hat A_{k+1}^\pm)^c \right| \right) \right] \leq  Q_3 \mu^k \left( \mu^{\kappa k} n^{-\kappa(1-2\eta)} + kn^{-\varepsilon} \right) ,$$
where $0 < \varepsilon < 1$ and $0 < \kappa \leq 1$ are those from Assumption \ref{A.Wasserstein}.
\end{lemma}

\begin{proof}
Let $\mathcal{F}_{m}$ denote the sigma-algebra generated by the bi-degree sequence and the history of the exploration process up until step $m-1$ is completed; note that this includes the value of $Z_m^\pm$. Define
$$u_{k+1} = \mathbb{E}_n \left[ 1(C_{k} \cap E_{k})  \left( \left|\hat A_{k+1}^\pm \cap (A_{k+1}^\pm)^c \right| + \left| A_{k+1}^\pm \cap (\hat A_{k+1}^\pm)^c \right| \right) \right]$$
and $\mathcal{E}(t)$ according to \eqref{eq:ErrorFunction}, then condition on $\mathcal{F}_{k}$ to obtain that
\begin{align}
u_{k+1} &= \mathbb{E}_n \left[ 1(C_{k} \cap E_{k}) E\left[ \left. \sum_{{\bf i} \in \hat A_{k}^\pm \cap A_{k}^\pm} \left( \hat \chi_{\bf i}^\pm - \chi_{\bf i}^\pm \right)^+ + \sum_{{\bf i} \in \hat A_{k}^\pm \cap (A_{k}^\pm )^c} \hat \chi_{\bf i}^\pm \right| \mathcal{F}_{k} \right] \right] \label{eq:ukFirst} \\
&\hspace{5mm} + \mathbb{E}_n \left[ 1(C_{k} \cap E_{k}) E\left[ \left. \sum_{{\bf i} \in A_{k}^\pm \cap \hat A_{k}^\pm} \left(  \chi_{\bf i}^\pm - \hat \chi_{\bf i}^\pm \right)^+ + \sum_{{\bf i} \in A_{k}^\pm \cap (\hat A_{k}^\pm )^c} \chi_{\bf i}^\pm \right| \mathcal{F}_{k} \right] \right] . \label{eq:ukSecond}
\end{align}
To analyze the conditional expectations we first note that for ${\bf i} \in \hat A_{k}^\pm \cap A_{k}^\pm$ we must have that $T_{\bf i}^\pm \leq Y_{k}^\pm$. Moreover, on the event $C_k \cap E_k$ we have that 
$$Y_k^\pm = \sum_{i=1}^k Z_i^\pm \leq (1+n^{-\gamma}) \sum_{i=1}^k \hat Z_i^\pm \leq 2 n^\eta \sum_{i=1}^k \mu^i \leq 2 \mu^{k+1} n^\eta/(\mu-1) \triangleq y_k,$$
which for the range of values of $k$ in the lemma satisfies $y_k = o(n)$ as $n \to \infty$. It follows by Lemma~\ref{L.KRdistance} and the tower property that on the event $C_k \cap E_k$, the conditional expectation in \eqref{eq:ukFirst} is bounded from above by
\begin{align*}
\sum_{{\bf i} \in \hat A_{k}^\pm \cap A_{k}^\pm} \mathbb{E}_n \left[ \left. \mathcal{E}(T_{\bf i}^\pm) \right| \mathcal{F}_{k} \right] + \left| \hat A_{k}^\pm \cap (A_{k}^\pm )^c \right| \mu &\leq \hat Z_{k}^\pm \mathbb{E}_n \left[ \left. \mathcal{E}(Y_{k}^\pm) \right| \mathcal{F}_{k} \right] + \left| \hat A_{k}^\pm \cap (A_{k}^\pm )^c \right| \mu, \\
&\leq \hat Z_{k}^\pm \mathbb{E}_n \left[ \left. \mathcal{E}(y_{k}) \right| \mathcal{F}_{k} \right] + \left| \hat A_{k}^\pm \cap (A_{k}^\pm )^c \right| \mu,
\end{align*}
where we used the observation that $\mathcal{E}(t)$ is non decreasing and $T_{\bf i}^\pm \leq Y_{k}^\pm \leq y_k$. 

Similarly, the conditional expectation in \eqref{eq:ukSecond} is bounded from above by
\begin{align*}
\mathbb{E}_n \left[ \left. \sum_{{\bf i} \in A_{k}^\pm} \left(  \chi_{\bf i}^\pm - \hat \chi_{\bf i}^\pm \right)^+ + \sum_{{\bf i} \in A_{k}^\pm \cap (\hat A_{k}^\pm )^c} \hat \chi_{\bf i}^\pm \right| \mathcal{F}_{k} \right] &\leq \sum_{{\bf i} \in A_{k}^\pm} \mathbb{E}_n \left[ \left. \mathcal{E}(T_{\bf i}^\pm) \right| \mathcal{F}_{k} \right] + \left| A_{k}^\pm \cap (\hat A_{k}^\pm )^c \right| \mu \\
&\leq Z_{k}^\pm \mathbb{E}_n \left[ \left. \mathcal{E}(Y_{k}^\pm) \right| \mathcal{F}_{k} \right] + \left| A_{k}^\pm \cap (\hat A_{k}^\pm )^c \right| \mu \\
&\leq (1+n^{-\gamma}) \hat Z_k^\pm \mathbb{E}_n \left[ \left. \mathcal{E}(y_{k}) \right| \mathcal{F}_{k} \right] + \left| A_{k}^\pm \cap (\hat A_{k}^\pm )^c \right| \mu,
\end{align*}
where in the last step we also used that $Z_k \leq (1+ n^{-\gamma}) \hat Z_k$ on the event $C_k$. 

Noting that $C_k \cap E_k \subseteq C_{k-1} \cap E_{k-1}$ gives 
\begin{align*}
u_{k+1} &\leq \mathbb{E}_n \left[ 1(C_k \cap E_k) (2 + n^{-\gamma}) \hat Z_k^\pm \mathbb{E}_n \left[ \left. \mathcal{E}(y_{k}) \right| \mathcal{F}_{k} \right] \right] + \mu u_k \\
&\leq 3 \mathbb{E}_n \left[ 1(E_k) \hat Z_k^\pm \mathcal{E}(y_k) \right] + \mu u_k. 
\end{align*}
By Lemma \ref{L.BoundE} we have that
\begin{align*}
3\mathbb{E}_n \left[ 1(E_k) \hat Z_k^\pm \mathcal{E}(y_k) \right] &\leq 3Q_2 \mu^k \left( \frac{y_k^\kappa}{n^{\kappa(1-\eta)}} + n^{-\varepsilon}  \right)  \\
&\leq Q_2' \left( \mu^{(1+\kappa) k} n^{-\kappa(1-2\eta)} + \mu^k n^{-\varepsilon} \right)
\end{align*}
for some constants $Q_2, Q_2' < \infty$. Let $r_k = Q_2' \mu^k \left( \mu^{\kappa k} n^{-\kappa(1-2\eta)} + n^{-\varepsilon} \right)$ and iterate the inequality $u_{k+1} \leq r_k + \mu u_k$ to obtain
\begin{align*}
u_{k+1} &\leq \sum_{j=1}^{k} \mu^{j-1} r_{k+1-j} + \mu^{k} u_1 = \sum_{j=1}^{k} \mu^{j-1} Q_2' \mu^{k+1-j}  \left( \mu^{\kappa(k+1-j)} n^{-\kappa(1-2\eta)} + n^{-\varepsilon} \right) + \mu^k u_1 \\
&= Q_2' \mu^k \left( \sum_{j=1}^k \mu^{\kappa(k+1-j)} n^{-\kappa(1-2\eta)} + k n^{-\varepsilon} \right)  + \mu^k u_1 \\
&= Q_2' \mu^k \left( n^{-\kappa(1-2\eta)} \frac{\mu^\kappa (\mu^{\kappa k} - 1)}{\mu^{\kappa} - 1}  + k n^{-\varepsilon}  \right) + \mu^k u_1.
\end{align*}
Noting that 
$$u_1 = \mathbb{E}_n \left[ |\hat \chi_{\emptyset}^\pm - \chi_\emptyset^\pm | \right] \leq n^{-\varepsilon}$$
completes the proof. 
\end{proof}

The last preliminary result before the proof of Theorem \ref{thm:main_coupling} is an analysis of the coupling when the branching process becomes extinct, which occurs most likely when the out- or in-component of the node being explored in the graph is small.

\begin{prop} \label{P.OnExtinction}
Fix $0 < \delta < 1$,  $0 < \gamma < \min\{\delta \kappa, \varepsilon\}$ and define for $k \geq 1$ the event $C_k$ according to \eqref{eq:DefCk}. Let $W^\pm = \lim_{k \to \infty} \hat Z_k^\pm/(\nu \mu^{k-1})$. Then, provided $({\bf D}_n^-, {\bf D}_n^+)$ satisfies Assumption~\ref{A.Wasserstein}, there exists a constant $Q_4 < \infty$ such that for all $1 \leq k \leq (1-\delta)\log n/\log \mu$, 
$$\mathbb{P}_n\left( C_{k}^c, \, W^\pm = 0 \right) \leq Q_4 n^{-\varepsilon \kappa/(1+\kappa+\varepsilon)}.$$
\end{prop}

\begin{proof}
We start by pointing out that if $q^\pm = 0$ the probability that the branching process $\{ \hat Z_k^\pm: k \geq 1\}$  becomes extinct is zero unless $\hat Z_1^\pm = 0$ (see Theorem 4 in \cite{Athreya2012}). Since in our construction of the coupling $Z_1^\pm = \hat Z_1^\pm$, we may assume from now on that $q^\pm > 0$. 

Analogously to the events $E_m$ and $C_m$ defined in \eqref{eq:DefEk} and \eqref{eq:DefCk}, we now define for $m \geq 1$
\begin{align}
B_m &= \bigcap_{r=1}^m \left\{ \left| \hat A^\pm_{r} \cap (A^\pm_{r})^c \right| =0, \, \left| A^\pm_{r} \cap (\hat A^\pm_{r})^c \right|=0 \right\}, \label{eq:DefBk} \\
I_m &= \bigcap_{r=1}^m \left\{  \hat Z_r^\pm/ (\lambda^\pm)^r \leq n^\tau \right\}, \notag
\end{align}
where $\lambda^\pm = \sum_{j=1}^\infty j f^\pm(j) (q^\pm)^{j-1} < 1$ and $\tau = \kappa/(1+\kappa +\varepsilon)$. Now note that
$$\mathbb{P}_n( C_{k}^c, \, W^\pm = 0) \leq \mathbb{P}_n( C_{k}^c \cap I_{k}) + P(I_{k}^c, \, W^\pm = 0),$$
where the last probability is independent of the bi-degree sequence. Next, use Lemma \ref{L.Extinction} to see that conditional on $\{ W^\pm = 0\}$, $\hat Z_k^\pm/ (\nu^\pm (\lambda^\pm)^{k-1})$ is a mean one martingale, where $\nu^\pm$ is the mean of $\tilde g^\pm$ and $\lambda^\pm$ is the mean of $\tilde f^\pm$. It then follows from Doob's inequality that
$$P(I_{k}^c, \, W^\pm = 0) \leq P(I_{k}^c | W^\pm = 0) \leq (\nu^\pm/ \lambda^\pm) n^{-\tau}.$$
Next, write
$$\mathbb{P}_n( C_{k}^c \cap I_{k}) \leq \mathbb{P}_n( B_{k}^c \cap I_{k}) \leq \sum_{r=1}^{k} \mathbb{P}_n( B_{r-1} \cap B_r^c \cap I_{k})$$
and note that the event $B_{r-1}$ implies that $Z^\pm_i = \hat Z_i^\pm$ for all $1 \leq i \leq r-1$. In addition, since $\lambda^\pm < 1$ we have $(\lambda^\pm)^{r_n+1} n^{-\tau} < 1$ for $r_n \triangleq \lfloor \tau \log n/|\log \lambda^\pm| \rfloor$. The observation that $\hat Z^\pm_i$ is integer valued then gives that the event $B_{r-1} \cap I_{k}$  implies that $Z^\pm_{r-1} = \hat Z_{r-1}^\pm = 0$ for all $r -1 > r_n$. Since for any $r \geq 1$ we have
\begin{align}
 \left| \hat A^\pm_{r} \cap (A^\pm_r)^c \right| &=  \sum_{{\bf i} \in \hat A_{r-1}^\pm \cap A_{r-1}^\pm} \sum_{j=1}^{\hat \chi_{\bf i}^\pm} 1(({\bf i},j) \in (A_r^\pm)^c)  +  \sum_{{\bf i} \in \hat A_{r-1}^\pm \cap (A_{r-1}^\pm)^c} \hat \chi_{\bf i}^\pm \notag \\
 &=  \sum_{{\bf i} \in \hat A_{r-1}^\pm \cap A_{r-1}^\pm} \left( \hat \chi_{\bf i}^\pm - \chi_{\bf i}^\pm \right)^+  +  \sum_{{\bf i} \in \hat A_{r-1}^\pm \cap (A_{r-1}^\pm)^c} \hat \chi_{\bf i}^\pm, \label{eq:Symmetry1}
 \end{align}
 and 
\begin{align}
 \left| A^\pm_{r} \cap (\hat A^\pm_r)^c \right|  &= \sum_{{\bf i} \in A_{r-1}^\pm \cap \hat A_{r-1}^\pm} \left( \chi_{\bf i}^\pm - \hat \chi_{\bf i}^\pm \right)^+   + \sum_{{\bf i} \in A_{r-1}^\pm \cap (\hat A_{r-1}^\pm)^c} \chi_{\bf i}^\pm \label{eq:Symmetry2} 
\end{align}
then, $\mathbb{P}_n( B_{r-1} \cap B_r^c \cap I_{k}) = 0$ for all $r_n+2 \leq r \leq k$. For $1 \leq r \leq r_n+1$ we obtain using \eqref{eq:Symmetry1} and \eqref{eq:Symmetry2} that
\begin{align*}
\mathbb{P}_n( B_{r-1} \cap B_r^c \cap I_{k}) &\leq \mathbb{P}_n\left( \left| \hat A^\pm_{r} \cap (A^\pm_{r})^c \right| +  \left| A^\pm_{r} \cap (\hat A^\pm_{r})^c \right| \geq 1, \, B_{r-1} \cap I_{r-1} \right) \\
&= \mathbb{P}_n\left( \sum_{{\bf i} \in \hat A^\pm_{r-1}} \left| \chi_{\bf i}^\pm - \hat \chi_{\bf i}^\pm \right| \geq 1, \, B_{r-1} \cap I_{r-1} \right).
\end{align*}
Now let $\mathcal{F}_r$ denote the sigma-algebra generated by the bi-degree sequence and the history of the exploration process up until step $r-1$ is completed; note that this includes the value of $Z_r^\pm$. Also define $\mathcal{G}_{\bf i}$ to be the sigma-algebra generated by the bi-degree sequence and the exploration process up to the time that outbound (inbound) stub {\bf i} is about to be traversed; note that for ${\bf i} \in \hat A_{r-1}^\pm$ we have $\mathcal{F}_{r-1} \subseteq \mathcal{G}_{\bf i}$. Applying Markov's inequality conditionally on $\mathcal{F}_{r-1}$ we obtain
$$\mathbb{P}_n\left( \sum_{{\bf i} \in \hat A^\pm_{r-1}} \left| \chi_{\bf i}^\pm - \hat \chi_{\bf i}^\pm \right| \geq 1, \, B_{r-1} \cap I_{r-1} \right) \leq \mathbb{E}_n \left[ 1(B_{r-1} \cap I_{r-1}) \sum_{{\bf i} \in \hat A_{r-1}^\pm} \mathbb{E}_n \left[ \left. \left| \chi_{\bf i}^\pm - \hat \chi_{\bf i}^\pm \right| \right| \mathcal{F}_{r-1} \right] \right].$$
To analyze the conditional expectation, recall that $T_{\bf i}^\pm$ is the number of stubs (outbound for $+$ and inbound for $-$) that have been seen up until it is stub ${\bf i}$'s turn to be traversed, and use Lemma~\ref{L.KRdistance} to obtain
$$\mathbb{E}_n \left[ \left. \left| \chi_{\bf i}^\pm - \hat \chi_{\bf i}^\pm \right| \right| \mathcal{F}_{r-1} \right]  = \mathbb{E}_n \left[ \left. \mathbb{E}_n \left[ \left. \left| \chi_{\bf i}^\pm - \hat \chi_{\bf i}^\pm \right|  \right| \mathcal{G}_{\bf i} \right] \right| \mathcal{F}_{r-1} \right] \leq  \mathbb{E}_n \left[ \left. \mathcal{E}(T_{\bf i}^\pm) \right| \mathcal{F}_{r-1} \right].$$
Recall that on the event $B_{r-1}$ we have $Z_i^\pm = \hat Z_i^\pm$ for all $1 \leq i \leq r-1$, and therefore, for ${\bf i} \in \hat A_{r-1}^\pm$,  $T_{\bf i}^\pm \leq Y_{r-1}^\pm = \hat Y_{r-1}^\pm = \sum_{i=1}^{r-1} \hat Z_i^\pm$. Moreover, on the event $I_{r-1}$ we have $\hat Y_{r-1}^\pm \leq \sum_{i=1}^{r-1} (\lambda^\pm)^{r-1} n^{\tau} \leq n^\tau/(1-\lambda^\pm)$, and it follows from Lemma \ref{L.BoundE} that
\begin{align*}
\mathbb{P}_n( B_{r-1} \cap B_r^c \cap I_{k}) &\leq \mathbb{E}_n\left[ 1(I_{r-1}) \sum_{{\bf i} \in \hat A_{r-1}^\pm} \mathbb{E}_n\left[ \mathcal{E}(n^\tau/(1-\lambda^\pm)) | \mathcal{F}_{r-1} \right] \right] \\
&= \mathbb{E}_n\left[ 1(I_{r-1}) \hat Z_{r-1}^\pm  \mathcal{E}(n^\tau/(1-\lambda^\pm))  \right] \\
&\leq (\lambda^\pm)^{r-1} n^\tau \mathbb{E}_n\left[ \mathcal{E}(n^\tau/(1-\lambda^\pm))  \right] \\
&\leq Q_2 (\lambda^\pm)^{r-1} n^\tau \left(  \frac{ (n^\tau/(1-\lambda^\pm))^\kappa}{n^\kappa} + n^{-\varepsilon } \right) \\
&= O\left(   (\lambda^\pm)^{r}  n^{-\varepsilon\tau} \right),
\end{align*}
where the last equality is due to our choice of $\tau$. Thus, we have shown that
\begin{align*}
\mathbb{P}_n\left( C_{k}^c, \, W^\pm = 0 \right) = O\left(  n^{-\varepsilon\tau} \sum_{r=1}^{r_n+1} (\lambda^\pm)^{r}  +  n^{-\tau}  \right) = O\left( n^{-\varepsilon\tau} \right).
\end{align*}
\end{proof}

We are now ready to prove Theorem \ref{thm:main_coupling}, which in view of Proposition \ref{P.OnExtinction}, reduces to analyzing the event that the error in the coupling is large conditionally on the branching process surviving. 

\begin{proof}[Proof of Theorem \ref{thm:main_coupling}]
Let $W^\pm = \lim_{k \to \infty} \hat Z_k^\pm/(\nu \mu^{k-1})$ and for each $m \geq 1$ define the event $C_m$ according to \eqref{eq:DefCk}.  Now note that
$$\mathbb{P}_n(C_{k}^c) \leq \mathbb{P}_n (C_{k}^c, \, W^\pm = 0) + \mathbb{P}_n (C_{k}^c, \, W^\pm > 0),$$
where by Proposition \ref{P.OnExtinction} we have 
$$\mathbb{P}_n (C_{k}^c, \, W^\pm = 0) \leq Q_4 n^{-\varepsilon\kappa/(1+\kappa+\epsilon)}$$
for some constant $Q_4 < \infty$. 

To analyze $\mathbb{P}_n( C_{k}^c, \, W^\pm > 0)$ we proceed similarly to the proof of Proposition \ref{P.OnExtinction} by setting $\eta = (\delta \kappa - \gamma)/(4\kappa) \in (0, \delta/4)$ and defining the events $E_m$ and $B_m$, $m \geq 1$, according to \eqref{eq:DefEk} and \eqref{eq:DefBk}, respectively. Define also $s_n = \min\{k, \lceil c \log n/\log \mu \rceil\}$, with $0 < c < \min\{\kappa(1-2\eta)/(1+\kappa), \, \varepsilon\}$, and note that $C_r^c \subset B_r^c$ for any $r \geq 1$. We start by deriving an upper bound as follows
\begin{align*}
\mathbb{P}_n( C_{k}^c, \, W^\pm > 0) &\leq \mathbb{P}_n( C_{k}^c \cap E_{k}, \, W^\pm > 0) + P(E_{k}^c) \\
&\leq \mathbb{P}_n (C_{s_n-1} \cap C_{k}^c \cap E_{k}, \, W^\pm > 0) + \mathbb{P}_n (C_{s_n-1}^c \cap E_{k}, \, W^\pm > 0) + P(E_{k}^c) \\
&\leq \mathbb{P}_n (C_{s_n-1} \cap C_{k}^c \cap E_{k}, \, W^\pm > 0) + \mathbb{P}_n (B_{s_n-1}^c \cap E_{k}) + P(E_{k}^c) \\
&\leq \sum_{r=1}^{s_n-1} \mathbb{P}_n( B_{r-1} \cap B_r^c \cap E_{k} ) + \mathbb{P}_n (C_{s_n-1} \cap C_{k}^c \cap E_{k}, \, W^\pm > 0) + P(E_{k}^c).
\end{align*}
Note that Doob's inequality gives $P(E_{k}^c) \leq (\mu/\nu) n^{-\eta}$. Also, if we let $\mathcal{F}_r$ denote the sigma-algebra generated by the bi-degree sequence and the history of the exploration process up until step $r-1$ is completed, the same steps used in the proof of Proposition \ref{P.OnExtinction} give that for $1 \leq r \leq s_n-1$, 
\begin{align*}
\mathbb{P}_n( B_{r-1} \cap B_r^c \cap E_{k} ) &\leq \mathbb{E}_n \left[ 1(E_{r-1}) \sum_{{\bf i} \in \hat A_{r-1}^\pm} \mathbb{E}_n [\mathcal{E}(\hat Y_{r-1}^\pm) | \mathcal{F}_{r-1}] \right] \\
&= \mathbb{E}_n\left[ 1(E_{r-1}) \hat Z_{r-1}^\pm  \mathcal{E}(\hat Y_{r-1}^\pm) \right] \leq  \mathbb{E}_n\left[ 1(E_{r-1}) \hat Z_{r-1}^\pm \mathcal{E}( n^\eta \mu^r/(\mu-1) ) \right] \\
&\leq Q_2 \mu^{r-1} \left( \frac{(n^\eta \mu^r/(\mu-1))^\kappa}{n^{\kappa(1-\eta)}} + n^{-\varepsilon} \right),
\end{align*}
where we used that on the event $E_{r-1}$ we have $\hat Y_{r-1}^\pm \leq \sum_{i=1}^{r-1} \mu^{i} n^\eta \leq n^\eta \mu^r/(\mu-1) $ and that $\mathcal{E}(t)$ is non decreasing, followed by an application of Lemma \ref{L.BoundE}. We then obtain that
$$\mathbb{P}_n( B_{r-1} \cap B_r^c \cap E_{k}) = O\left( \frac{\mu^{r(1+\kappa)}}{n^{\kappa(1-2\eta)}} + \mu^r n^{-\varepsilon } \right),$$
which implies that, 
\begin{align*}
\sum_{r=1}^{s_n-1} \mathbb{P}_n( B_{r-1} \cap B_r^c \cap E_{k} ) &= O\left( \frac{\mu^{s_n(1+\kappa)}}{n^{\kappa(1-2\eta)}} + \mu^{s_n} n^{-\varepsilon } \right) = O\left( n^{c(1+\kappa)- \kappa(1-2\eta)} + n^{c-\varepsilon} \right).
\end{align*}

We have thus shown that, as $n \to \infty$, 
$$\mathbb{P}_n( C_{k}^c, \, W^\pm > 0) \leq \mathbb{P}_n( C_{s_n-1} \cap C_{k}^c \cap E_{k}, \, W^\pm > 0) + O\left( n^{-\eta} + n^{c(1+\kappa) - \kappa(1-2\eta)} + n^{c-\varepsilon } \right),$$
with all the exponents of $n$ inside the big-O term strictly negative. To analyze the remaining probability we first introduce one last conditioning event. Set $1 < u = \mu^{1-b}$ with $b = \min\{ (1-\delta)/2, \, (\varepsilon - \gamma)/2, \, (\kappa\delta - \gamma)/4 \} /(1-\delta) \in (0,1/2)$, and define
$$J_{s_n} = \left\{ \inf_{r \geq s_n} \hat Z_{r}^\pm/u^r \geq 1 \right\}.$$
Now write
\begin{align*}
\mathbb{P}_n( C_{s_n-1} \cap C_{k}^c \cap E_{k}, \, W^\pm > 0) &\leq \mathbb{P}_n( C_{s_n-1} \cap C_{k}^c \cap E_{k} \cap J_{s_n} ) + P(J_{s_n}^c, \, W^\pm > 0) \\
&\leq \sum_{r=s_n}^{k} \mathbb{P}_n( C_{r-1} \cap C_r^c \cap E_{k} \cap J_{s_n}) + P(J_{s_n}^c, \, W^\pm > 0).
\end{align*}
By Lemma \ref{L.Martingale} we have
$$P(J_{s_n}^c, \, W^\pm > 0) \leq Q_1 \left( u^{-\kappa s_n} + (u/\mu)^{\alpha^\pm s_n} 1(q^\pm > 0) \right) = O\left( n^{-\kappa c (1-b)} + n^{-\alpha^\pm c b} 1( q^\pm > 0) \right),$$
where $\alpha^\pm = |\log \lambda^\pm|/\log \mu > 0$. 

To bound each of the remaining probabilities, $\mathbb{P}_n( C_{r-1} \cap C_r^c \cap E_{k} \cap J_{s_n})$, use the union bound followed by Markov's inequality applied conditionally on $\mathcal{F}_{k-1}$, to obtain, for $s_n \leq r \leq k$, 
\begin{align*}
\mathbb{P}_n( C_{r-1} \cap C_r^c \cap E_{k} \cap J_{s_n}) &\leq \mathbb{P}_n\left( \left| \hat A_r^\pm \cap (A_r^\pm)^c \right| > \hat Z_r^\pm n^{-\gamma}, \, C_{r-1} \cap E_{k} \cap J_{s_n}\right) \\
&\hspace{5mm} + \mathbb{P}_n\left( \left| A_r^\pm \cap (\hat A_r^\pm)^c \right| > \hat Z_r^\pm n^{-\gamma}, \, C_{r-1} \cap E_{k} \cap J_{s_n}\right) \\
&\leq \mathbb{P}_n\left( \left| \hat A_r^\pm \cap (A_r^\pm)^c \right| > u^r n^{-\gamma}, \, C_{r-1} \cap E_{r-1} \right) \\
&\hspace{5mm} + \mathbb{P}_n\left( \left| A_r^\pm \cap (\hat A_r^\pm)^c \right| > u^r n^{-\gamma}, \, C_{r-1} \cap E_{r-1} \right) \\
&\leq \frac{n^\gamma}{u^r} \mathbb{E}_n \left[ 1(C_{r-1} \cap E_{r-1}) \mathbb{E}_n\left[ \left. \left| \hat A_r^\pm \cap (A_r^\pm)^c \right| + \left| A_r^\pm \cap (\hat A_r^\pm)^c \right| \right| \mathcal{F}_{r-1} \right] \right].
\end{align*}
It follows from Lemma \ref{L.Expectations} that
\begin{align*}
\mathbb{P}_n( C_{r-1} \cap C_r^c \cap E_{k} \cap J_{s_n}) &\leq Q_3 \frac{n^\gamma}{u^r} \cdot \mu^r \left( \mu^{\kappa r} n^{-\kappa(1-2\eta)} + r n^{-\varepsilon} \right),
\end{align*}
which in turn implies that, as $n \to \infty$,
\begin{align*}
\sum_{r=s_n}^{k} \mathbb{P}_n( C_{r-1} \cap C_r^c \cap E_{k} \cap J_{s_n}) &= O\left( n^{\gamma-\kappa(1-2\eta)} \sum_{r=s_n}^{k} (\mu^{1+\kappa}/u)^{r} + n^{\gamma-\varepsilon} \sum_{r=s_n}^{k} r (\mu/u)^r  \right) \\
&= O\left( n^{\gamma-\kappa(1-2\eta)} (\mu^{1+\kappa}/u)^{k} + n^{\gamma - \varepsilon} k (\mu/u)^{k} \right) \\
&= O\left( n^{\gamma -\kappa(1-2\eta) + (1-\delta) (b+\kappa)} +  n^{\gamma-\varepsilon + (1-\delta)b} \log n \right) \\
&= O\left( n^{(\gamma-\kappa\delta)/2 + (1-\delta)b} + n^{\gamma-\varepsilon + (1-\delta)b} \log n \right).
\end{align*}
Since all the exponents inside the big-O term are strictly negative, this completes the proof. 
\end{proof}


\subsection{Distances in the directed configuration model}
\label{SS.ProofsDistances} 

In this section we prove Proposition \ref{P.HopCount}, Proposition \ref{P.BPapprox}, and Corollary \ref{C.FiniteHopcount}. As  mentioned in Section \ref{S.Distances},  Propositions \ref{P.HopCount} and \ref{P.BPapprox} together yield Theorem \ref{T.MainTheoremBody}. As a preliminary result for the proof of Proposition \ref{P.HopCount} we first state and prove the following technical lemma. Throughout the remainder of the section we use $x \wedge y$ to denote the minimum of $x$ and $y$.  

\begin{lemma} \label{L.Taylor}
For any nonnegative $x, x_0 > 0$, $y_i, z_i \geq 0$ with $z_i < x$ for all $i$, and any $m \geq 1$, we have 
$$- \frac{x_0}{x^2} (x_0-x)^+ - \frac{x_0}{2} \max_{1 \leq i \leq m} \frac{z_i}{(x-z_i)^2} \leq \prod_{i=1}^{m} \left( 1- \frac{z_i}{x} \right)^{y_i} - \exp\left\{ - \frac{1}{x_0} \sum_{i=1}^{m} y_i z_i \right\} \leq \frac{|x-x_0|}{(x \wedge x_0)}.$$
\end{lemma}

\begin{proof}
For the upper bound note that
\begin{align*}
\prod_{i=2}^{m} \left( 1- \frac{z_i}{x} \right)^{y_i} &= \exp \left\{ \sum_{i=1}^{m} y_i \log \left(1 - \frac{z_i}{x} \right) \right\} \leq \exp\left\{ - \frac{1}{x} \sum_{i=1}^{m} y_i z_i \right\}  \\
&\leq \exp\left\{ - \frac{1}{x_0} \sum_{i=1}^{m} y_i z_i \right\} + \frac{|x-x_0|}{x_0 \wedge x} , 
\end{align*}
where in the second step we used the inequality $\log(1-t) \leq -t$ for $t \in [0, 1)$ and in the third step we used the inequality 
$$|e^{-S/x} - e^{-S/x_0}| \leq \frac{S}{\xi^2} e^{-S/\xi} |x - x_0| \leq \sup_{t \geq 0} t e^{-t} \cdot \frac{|x-x_0|}{\xi} \leq \frac{|x-x_0|}{(x \wedge x_0)}$$ 
for any $S, x, x_0 \geq 0$ and some $\xi$ between $x$ and $x_0$.  

Similarly, using the first order Taylor expansion for $\log (1 - t)$ we obtain
$$\log(1-c/x) = -\frac{c}{x} - \frac{c^2}{2 x^2 (1-\xi')^2}  = - \frac{c}{x_0} + \frac{c}{(\xi'')^2}(x-x_0) - \frac{c^2}{2 x^2 (1-\xi')^2}$$
for any $c < x$, $0 < \xi' < c/x$, and $\xi''$ between $x$ and $x_0$, which in turn yields the inequality
$$\log(1-c/x) \geq -\frac{c}{x_0} - \frac{c}{x^2} (x_0-x)^+ - \frac{c^2}{2(x- c)^2}.$$
It follows that 
\begin{align*}
\prod_{i=1}^{m} \left( 1- \frac{z_i}{x} \right)^{y_i} &\geq \exp\left\{ -\sum_{i=1}^{m} \left( \frac{y_i z_i}{x_0} + \frac{y_i z_i}{x^2} (x_0-x)^+ + \frac{y_i z_i^2}{2(x- z_i)^2} \right) \right\}  \\
&\geq \exp\left\{ - \frac{1}{x_0} \sum_{i=1}^{m} y_i z_i \right\} -  \exp\left\{ - \frac{1}{x_0} \sum_{i=1}^{m} y_i z_i \right\}  \sum_{i=1}^{m} \left( \frac{y_i z_i}{x^2} (x_0-x)^+ + \frac{y_i z_i^2}{2(x- z_i)^2} \right)  \\
&\geq \exp\left\{ - \frac{1}{x_0} \sum_{i=1}^{m} y_i z_i \right\} - \frac{x_0}{x^2} (x_0 - x)^+ - \frac{x_0}{2} \max_{1\leq i \leq m} \frac{z_i}{(x-z_i)^2} .
\end{align*}
\end{proof}

We are now ready to prove Proposition~\ref{P.HopCount}. 

\begin{proof}[Proof of Proposition~\ref{P.HopCount}]
Let $0 < \gamma < \min\{\kappa \delta, \varepsilon\}$, and construct the pairs of processes 
$\{ (\Zout_i, \tbtZout_i): i \geq 0 \}$ and $\{ (\Zin_i, \tbtZin_i): i \geq 0\}$ according to the 
coupling described in Section \ref{SS.Coupling}, independently of each other.  Now define the event
\[
	\mathcal{E}_k = \bigcap_{m=1}^{\lceil k/2 \rceil+1} \left\{ 
	\hat Z_m^+ \left(1 -  n^{-\gamma} \right) \leq Z^+_m 
	\leq \hat Z_m^+ \left(1 + n^{-\gamma} \right), \, 
	\hat Z_m^- \left(1 - n^{-\gamma} \right) \leq Z^-_m 
	\leq \hat Z_m^- \left(1 + n^{-\gamma} \right) \right\}
\]
and note that since $\{(\Zout_i, \hat Z_i^{+}): i \geq 0 \}$ and 
$\{(\Zin_i, \hat Z_i^-): i \geq 0\}$ are independent, then Corollary~\ref{C.MainCoupling} gives 
$\mathbb{P}_n (\mathcal{E}_k^c) = O(n^{-a_1})$ for some $a_1 > 0$.  

Next, use the triangle inequality to get
\begin{align}
	&\left| \mathbb{P}_n( H_n > k) - E\left[\exp\left\{ -\frac{\sum_{i=2}^{k+1} 
		\tbtZout_{\ceil{i/2}} \tbtZin_{\floor{i/2}}}{\nu n} \right\} \right] \right| \notag \\
	&\leq \left| \mathbb{P}_n( H_n > k) -  \mathbb{E}_n \left[\exp\left\{-\frac{\sum_{i=2}^{k+1} 
		\Zout_{\ceil{i/2}} \Zin_{\floor{i/2}}}{\nu n} \right\} \right] \right| 
		\label{eq:hopcount_triangle1}\\
	&\hspace{5mm} + \left| \mathbb{E}_n \left[ \exp\left\{ - \frac{\sum_{i=2}^{k+1} 
		\Zout_{\ceil{i/2}} \Zin_{\floor{i/2}}}{\nu n} \right\} \right] - E\left[ \exp\left\{ 
		-\frac{\sum_{i=2}^{k+1} \tbtZout_{\ceil{i/2}} \tbtZin_{\floor{i/2}}}{\nu n} \right\}  
		\right] \right|. \label{eq:hopcount_triangle2}
\end{align}

We start by bounding~\eqref{eq:hopcount_triangle2}, for which we use the independence of 
$\{ \hat Z_i^+\}$ and $\{ \hat Z_i^-\}$ from the bi-degree sequence $({\bf D}_n^-, {\bf D}_n^+)$, and 
the inequality $| e^{-x} - e^{-y} | \leq e^{-(x \wedge y)} |x-y|$ for $x,y \geq 0$ to obtain
\begin{align*}
&\left|\mathbb{E}_n \left[\exp\left\{ -\frac{\sum_{i=2}^{k+1} \Zout_{\ceil{i/2}} 
	\Zin_{\floor{i/2}}}{\nu n} \right\} \right] - E\left[ \exp\left\{ -\frac{\sum_{i=2}^{k+1} 
	\tbtZout_{\ceil{i/2}} \tbtZin_{\floor{i/2}}}{\nu n} \right\}  \right] \right| \\
&\leq \mathbb{E}_n \left[ \left| \exp\left\{-\frac{\sum_{i=2}^{k+1} \Zout_{\ceil{i/2}} 
	\Zin_{\floor{i/2}}}{\nu n} \right\} - \exp\left\{-\frac{\sum_{i=2}^{k+1} \tbtZout_{\ceil{i/2}} 
	\tbtZin_{\floor{i/2}}}{\nu n} \right\} \right| \right]  \\
&\leq \frac{1}{n\nu} \mathbb{E}_n\left[ 1(\mathcal{E}_k) 
	\exp\left\{-\frac{S_k (1 - n^{-\gamma})^2}{\nu n} \right\} \left| \sum_{i=2}^{k+1} 
	\left(\Zout_{\ceil{i/2}} \Zin_{\floor{i/2}} -  \tbtZout_{\ceil{i/2}} \tbtZin_{\floor{i/2}} 
	\right) \right| \right] + \mathbb{P}_n (\mathcal{E}_k^c),
\end{align*}
where $S_k = \sum_{i=2}^{k+1}  Z^+_{\ceil{i/2}}  Z^-_{\floor{i/2}}$. Since on the event 
$\mathcal{E}_k$ we have that for all $2 \leq i \leq k+1$,
\[
	(1 - 2n^{-\gamma} + n^{-2\gamma}) \tbtZout_{\ceil{i/2}} \tbtZin_{\floor{i/2}} \leq 
	\Zout_{\ceil{i/2}} \Zin_{\floor{i/2}} \leq (1+2 n^{-\gamma} + n^{-2\gamma}) 
	\tbtZout_{\ceil{i/2}} \tbtZin_{\floor{i/2}},
\]
then for all $n \geq 1$,
\[ 
	\left| \Zout_{\ceil{i/2}} \Zin_{\floor{i/2}} - \tbtZout_{\ceil{i/2}} \tbtZin_{\floor{i/2}} 
	\right| \leq 3 n^{-\gamma} \tbtZout_{\ceil{i/2}} \tbtZin_{\floor{i/2}} \leq \frac{3n^{-\gamma}}{(1-n^{-\gamma})^2} Z_{\lceil i/2 \rceil}^+ Z_{\lfloor i/2 \rfloor}.
\]

It follows that \eqref{eq:hopcount_triangle2} is bounded from above by
\begin{align*}
	\frac{3 n^{-\gamma}}{(1-n^{-\gamma})^2}  \mathbb{E}_n \left[ 
		\exp\left\{ - \frac{S_k (1-n^{-\gamma})^2}{\nu n}  \right\}  \frac{1}{n\nu} S_k  \right] 
		+ \mathbb{P}_n (\mathcal{E}_k^c) \leq \frac{3 e^{-1} n^{-\gamma}}{(1-n^{-\gamma})^4}  
		+ \mathbb{P}_n(\mathcal{E}_k^c),
\end{align*}
where we used the observation that $\sup_{x \geq 0} e^{-x} x = e^{-1}$. 

We now proceed to bound~\eqref{eq:hopcount_triangle1}. From \eqref{eq:hopcount_main}, we have that
\[
	\mathbb{P}_n(H_n > k) = \mathbb{E}_n \left[ 1(\mathscr{S}_{k+1} \leq L_n) \prod_{i=2}^{k+1} 
	\prod_{s=0}^{\Zout_{\ceil{i/2}} - 1} \left( 1 - \frac{\Zin_{\floor{i/2}}}{L_n  - \mathscr{S}_{i-2} - s} 
	\right) \right],
\]
where $\mathscr{S}_k$ is as in~\eqref{eq:Bk}. Recall that from Assumption \ref{A.Wasserstein} we have 
$|L_n - n\nu| \leq n^{1-\varepsilon}$, and therefore $\{ \mathscr{S}_{k+1} \leq n^{b} \} \subseteq 
\{ \mathscr{S}_{k+1} \leq L_n\}$ for any $1-\delta < b < 1$ and all sufficiently large $n$. Now note that 
\begin{align*}
	\mathbb{E}_n \left[ 1(\mathscr{S}_{k+1} \leq n^{b}) \prod_{i=2}^{k+1} \left(1 - 
		\frac{\Zin_{\floor{i/2}}}{L_n - \mathscr{S}_{k+1}} \right)^{\Zout_{\ceil{i/2}}} \right] 
	&\leq \mathbb{P}_n(H_n > k) \leq \mathbb{E}_n \left[ \prod_{i=2}^{k+1} \left(1 - 
		\frac{\Zin_{\floor{i/2}}}{L_n} \right)^{\Zout_{\ceil{i/2}}} \right].
\end{align*}

Using Lemma \ref{L.Taylor} with $x = L_n$ and $x_0 = \nu n$ gives
\[
	\mathbb{P}_n(H_n > k) \leq \mathbb{E}_n \left[ e^{-S_k /(\nu n) } \right] 
	+ \frac{|L_n - \nu n|}{L_n \wedge (\nu n)} = \mathbb{E}_n \left[ e^{-S_k /(\nu n) } \right] 
	+ O(n^{-\varepsilon}).
\]
Similarly, using Lemma \ref{L.Taylor} with $x = L_n - \mathscr{S}_{k+1}$ and $x_0 = \nu n$ we obtain
\begin{align*}
	\mathbb{P}_n( H_n > k) &\geq \mathbb{E}_n\left[ e^{-S_k/(\nu n)} 1(\mathscr{S}_{k+1} \leq n^{b}) \right] \\
	&\hspace{5mm} -\mathbb{E}_n\left[ 1(\mathscr{S}_{k+1} \leq n^{b}) \frac{\nu n}{(L_n - \mathscr{S}_{k+1})^2} 
		(\nu n - L_n + \mathscr{S}_{k+1})^+ \right] \\
	&\hspace{5mm} - \mathbb{E}_n\left[ 1(\mathscr{S}_{k+1} \leq n^{b}) \frac{\nu n}{2} \max_{2 \leq i \leq k+1} 
		\frac{\Zin_{\floor{i/2}}}{(L_n - \mathscr{S}_{k+1} - \Zin_{\floor{i/2}})^2} \right] \\
	&\geq  \mathbb{E}_n\left[ e^{-S_k/(\nu n)} \right] - \mathbb{P}_n(\mathscr{S}_{k+1} > n^{b}) 
		- \frac{(n^{1-\varepsilon}+ n^{b})^+}{(L_n - n^{b})^2} 
		- \frac{\nu n}{2} \cdot \frac{n^{b}}{(L_n - 2n^{b})^2} \\
	&= \mathbb{E}_n\left[ e^{-S_k/(\nu n)} \right] - \mathbb{P}_n(\mathscr{S}_{k+1} > n^{b}) - O( n^{-(1-b)} ).
\end{align*}
Finally, note that using Markov's inequality we obtain
\begin{align*}
	\mathbb{P}_n( \mathscr{S}_{k+1} > n^{b}) &\leq \mathbb{P}_n(\mathcal{E}_k^c) 
		+ \mathbb{P}_n( \mathscr{S}_{k+1} > n^{b}, \mathcal{E}_k) \\
	&\leq \mathbb{P}_n(\mathcal{E}_k^c) + P\left( \sum_{j=1}^{\ceil{k/2}} 
		(\tbtZout_j + \tbtZin_j)(1 + n^{-\gamma}) > n^{b} \right) \\
	&\leq \mathbb{P}_n(\mathcal{E}_k^c) + \frac{1 + n^{-\gamma}}{n^{b}} \sum_{j=1}^{\ceil{k/2}} 
		E\left[ \tbtZout_j + \tbtZin_j \right] = O( n^{-a_1} + \mu^{\ceil{k/2}} n^{-b}),
\end{align*}
where in the last step we used the observation that $\sum_{j=1}^{\ceil{k/2}} E[ \tbtZout_j + 
\tbtZin_j]  = 2 \sum_{j=1}^{\ceil{k/2}} \nu \mu^{j-1} = O( \mu^{\ceil{k/2}})$. Since $k \leq 
2(1-\delta)\log n/(\log \mu)$, then $\mu^{\ceil{k/2}} = O(n^{1-\delta})$ and the result follows.
\end{proof}

We now proceed to prove Proposition~\ref{P.BPapprox}, which shows that the expression derived for the hopcount in Proposition~\ref{P.HopCount} can be closely approximated by an expression in terms of the limiting martingales of the branching processes $\{\hat Z_k^+: k \geq 0\}$ and $\{\hat Z_k^-: k \geq 0\}$. Note that this result is independent of the bi-degree sequence $({\bf D}_n^-, {\bf D}_n^+)$, since it involves only the coupled branching processes.

\begin{proof}[Proof of Proposition~\ref{P.BPapprox}]
Fix $0 < \epsilon < \kappa/(2+2\kappa)$  and set $m_n = \floor{(1-\epsilon) \log n/\log \mu}$. We 
start by noting that for $1 \leq k \leq m_n$ the inequality $|e^{-x} - e^{-y}| \leq |x - y|$ for 
$x,y \geq 0$ gives
\begin{align*}
	&\left| E\left[ \exp\left\{-\frac{1}{\nu n} \sum_{i=2}^{k+1} \tbtZout_{\ceil{i/2}} 
		\tbtZin_{\floor{i/2}} \right\} - \exp\left\{-\frac{\nu \mu^{k}}{(\mu-1)n} W^+ W^- \right\} 
		\right] \right| \\
	&\leq E\left[ \left| \frac{1}{\nu n} \sum_{i=2}^{k+1} \tbtZout_{\ceil{i/2}} \tbtZin_{\floor{i/2}} 
		- \frac{\nu \mu^{k}}{(\mu-1)n} W^- W^+ \right| \right] \\
	&\leq \frac{1}{\nu n} \sum_{i=2}^{k+1} E\left[ \tbtZout_{\ceil{i/2}}\right] 
		E\left[ \tbtZin_{\floor{i/2}} \right] + \frac{\nu \mu^{k}}{(\mu-1)n} E[W^+] E[W^-] \\
	&= \frac{1}{\nu n} \sum_{i=2}^{k+1} \nu^2 \mu^{\ceil{i/2} + \floor{i/2} - 2} 
		+ \frac{\nu \mu^{k}}{(\mu-1)n} E[W^+] E[W^-]  \\
	&= \frac{\nu}{\mu^2n} \left( \sum_{i=2}^{k+1} \mu^i + \frac{\mu^{k+2}}{\mu-1} \right) 
		\leq \frac{2\nu \mu^{k}}{n(\mu-1)} \leq \frac{2\nu \mu^{m_n}}{n(\mu-1)} 
		= O\left( n^{-\epsilon} \right)
\end{align*}
as $n \to \infty$, where in the second equality we used the observation that $\ceil{i/2} + 
\floor{i/2} = i$ for all $i \in \mathbb{N}$, and $E[W^-] = E[W^+] = 1$ (since $f^+$ and $f^-$ have 
finite moments of order $1+\kappa$). It remains to consider the case $k > m_n$. 

Suppose from now on that $k > m_n$ and note that 
\[
	\frac{\mu^{k} }{\mu-1}  = \sum_{i=m_n}^{k+1} \mu^{i - 2} + \frac{\mu^{m_n - 2}}{\mu-1},
\]
and therefore, using $|e^{-x} - e^{-y}| \leq |x-y|$ for $x,y \geq 0$ and the triangle inequality 
we obtain
\begin{align}
	&\left| E\left[ \exp\left\{-\frac{1}{\nu n} \sum_{i=2}^{k+1} \tbtZout_{\ceil{i/2}} 
		\tbtZin_{\floor{i/2}} \right\} - \exp\left\{-\frac{\nu \mu^{k}}{(\mu-1)n} W^+ W^- \right\} 
		\right] \right| \notag \\
	&\leq  E\left[ \left| \exp\left\{-\frac{1}{\nu n} \sum_{i=m_n}^{k+1} \tbtZout_{\ceil{i/2}} 
		\tbtZin_{\floor{i/2}} \right\} - \exp\left\{-\frac{\nu}{n} \sum_{i=m_n}^{k+1} \mu^{i - 2} W^+
		W^- \right\} \right| \right]  \label{eq:TwoExp} \\
	&\hspace{5mm} + E\left[ \frac{1}{\nu n} \sum_{i=2}^{m_n-1}  \tbtZout_{\ceil{i/2}} 
		\tbtZin_{\floor{i/2}} + \frac{\nu \mu^{m_n - 2}}{(\mu-1)n} W^+ W^- \right], \label{eq:SmallKs}
\end{align}
where $\eqref{eq:SmallKs}$ is of order $O(n^{-\epsilon})$ as shown above. To bound 
\eqref{eq:TwoExp}, let $W_k^\pm = \hat Z_k^\pm/(\nu \mu^{k-1})$, set $\eta = \epsilon/2$, and 
define the events $\mathcal{B} = \{ W^+ W^- > 0\}$,
\[
	\mathcal{C}_r^\pm = \left\{  \max_{\lfloor m_n/2 \rfloor \leq j \leq r} 1(W_j^\pm > 0) 
	|W_j^\pm - W^\pm|/ W_j^\pm \leq n^{-\eta} \right\}, \quad \text{for } r \geq \lfloor m_n/2 \rfloor,
\]
and $\mathcal{C}_k = \mathcal{C}_{\lfloor (k+1)/2 \rfloor}^+ \cap \mathcal{C}_{\lceil (k+1)/2 \rceil}^-$. Next, use the 
inequality $|e ^{-y} - e^{-x}| \leq e^{-(x \wedge y)} |x -y|$ for $x,y \geq 0$ to obtain
\begin{align*}
	& E\left[ 1(\mathcal{B} \cap \mathcal{C}_k) \left| \exp\left\{ - \frac{1}{\nu n}  \sum_{i = m_n}^{k+1}  
		\tbtZout_{\ceil{i/2}} \tbtZin_{\floor{i/2}} \right\} - \exp\left\{-\frac{\nu}{n} 
		\sum_{i = m_n}^{k+1} \mu^{i - 2} W^+ W^-  \right\} \right| \right] \\
	&\leq E\left[ 1(\mathcal{B} \cap \mathcal{C}_k) \exp\left\{-\frac{\nu}{\mu^2 n}  \sum_{i = m_n}^{k+1} 
		\mu^i W_{\ceil{i/2}}^+ W_{\floor{i/2}}^- (1-n^{-\eta}) \right\} \right. \\
	&\hspace{15mm} \left. \times \left| \frac{\nu}{\mu^2 n}  \sum_{i = m_n}^{k+1} \mu^i  \left( 
		W_{\ceil{i/2}}^+ W_{\floor{i/2}}^-  - W^+ W^- \right)   \right| \right] \\
	&\leq \frac{3 \nu n^{-\eta}}{\mu^2 n} E\left[ \exp\left\{ -\frac{\nu}{\mu^2 n} 
		\sum_{i = m_n}^{k+1} \mu^i W_{\ceil{i/2}}^+ W_{\floor{i/2}}^- (1-n^{-\eta}) \right\} 
		\sum_{i = m_n}^{k+1} \mu^i W_{\ceil{i/2}}^+ W_{\floor{i/2}}^- \right] \\
	&\leq \frac{3 n^{-\eta}}{\mu^2(1-n^{-\eta})} \sup_{x > 0} e^{-x} x \leq \frac{3 n^{-\eta}}{\mu^2(1-n^{-\eta})},
\end{align*}
where we have used the observation that on the event $\mathcal{B} \cap \mathcal{C}_k$ we have
\begin{align*}
	\left| W_{\ceil{i/2}}^+ W_{\floor{i/2}}^-  - W^+ W^-\right| 
	&\leq \left| W_{\ceil{i/2}}^+ - W^+ \right| W_{\floor{i/2}}^- 
		+ W^+ \left| W_{\floor{i/2}}^- - W^- \right| \\
	&\leq W_{\ceil{i/2}}^+ W_{\floor{i/2}}^- n^{-\eta} 
		+ (1+ n^{-\eta}) n^{-\eta} W_{\ceil{i/2}}^+ W_{\floor{i/2}}^- \\
	&\leq 3 n^{-\eta} W_{\ceil{i/2}}^+ W_{\floor{i/2}}^-,
\end{align*}
and that $\sup_{x \geq 0} x e^{-x} = e^{-1}$. Also, by using that $|e^{-x} - e^{-y}| \leq 1$ for 
$x,y \geq 0$ we obtain
\begin{align*}
	&E\left[ 1(\mathcal{B} \cap \mathcal{C}_k^c) \left| \exp\left\{ - \frac{1}{\nu n}  \sum_{i = m_n}^{k+1}  
		\tbtZout_{\ceil{i/2}} \tbtZin_{\floor{i/2}} \right\} - \exp\left\{-\frac{\nu}{\mu^2 n} 
		\sum_{i = m_n}^{k+1} \mu^{i} W^+ W^- \right\} \right| \right] \\
	&\leq P\left(W^+ > 0, ( \mathcal{C}_{\ceil{(k+1)/2}}^+)^c\right) 
		+ P\left(W^- > 0, \, ( \mathcal{C}_{\floor{(k+1)/2}}^-)^c\right). 
\end{align*}

To bound the last two probabilities set $u = \mu^{\epsilon (2+\kappa)/(\kappa -\kappa\epsilon)} 
\in (1,\mu)$, define the event $\mathcal{D}^\pm = \{ \inf_{j \geq \lfloor m_n/2 \rfloor} 
\hat Z_j^\pm/ u^{j} \geq 1 \}$, and note that for any $r \geq \floor{m_n/2}$,
\begin{align*}
	P\left(W^\pm > 0, \, (\mathcal{C}_r^\pm)^c \right) 
	&\leq P(W^\pm > 0, \, (\mathcal{C}_r^\pm)^c \cap \mathcal{D}^\pm ) + P(W^\pm > 0, (\mathcal{D}^\pm)^c) \\
	&\leq \sum_{j = \floor{m_n/2}}^r P( |W_j^\pm - W^\pm| > n^{-\eta} W_j^\pm, \, \mathcal{D}^\pm) \\
	&\hspace{10pt}+ P(W^\pm > 0, \, (\mathcal{D}^\pm)^c).
\end{align*}
By Lemma \ref{L.Martingale} we have that
\[
	P(W^\pm > 0, (\mathcal{D}^\pm)^c) \leq Q_1 \left( u^{-\kappa \floor{m_n/2}} 
	+ (u/\mu)^{\alpha^\pm \floor{m_n/2}} 1(q^\pm > 0) \right),
\]
for some constant $Q_1 < \infty$, where $\lambda^\pm = \sum_{i=1}^\infty i f^\pm(i) (q^\pm)^{i-1} \in 
[0, 1)$ and $\alpha^\pm = |\log \lambda^\pm|/\log\mu > 0$ if $q^\pm > 0$; and for the remaining 
probabilities we use the representation \eqref{eq:ConvergenceRate} for $W^\pm - W_j^\pm$ and Lemma 
\ref{L.Burkholder}, applied conditionally on $\hat Z_j^\pm$, to obtain
\begin{align*}
	P( |W_j^\pm - W^\pm| > n^{-\eta} W_j^\pm, \, \mathcal{D}^\pm) 
	&\leq P\left( \left| \sum_{{\bf i} \in \hat A_j^\pm} (\mathcal{W}_{\bf i}^\pm -1) \right| > 
		n^{-\eta} \hat Z_j^\pm, \, \hat Z_j^\pm \geq u^j \right) \\
	&\leq E\left[ 1(\hat Z_j^\pm \geq u^j) P\left( \left( \left| \sum_{{\bf i} \in \hat A_j^\pm} 
		(\mathcal{W}_{\bf i}^\pm -1) \right| > n^{-\eta} \hat Z_j^\pm \right| \hat Z_j^\pm \right) 
		\right] \\
	&\leq Q_{1+\kappa} E[|\mathcal{W}^\pm -1|^{1+\kappa}] \cdot E\left[ 1(\hat Z_j^\pm \geq u^j) 
		\frac{\hat Z_j^\pm}{(\hat Z_j^\pm n^{-\eta})^{1+\kappa} }  \right] \\
	&\leq Q_{1+\kappa} E[|\mathcal{W}^\pm-1|^{1+\kappa}] \frac{n^{\eta(1+\kappa)}}{u^{\kappa j}},
\end{align*}
where $\{ \mathcal{W}_{\bf i}^\pm\}$ are i.i.d.~random variables having the same distribution as 
$\mathcal{W}^\pm = \lim_{r \to \infty} \mathcal{Z}^\pm_r/\mu^r$, and $Q_{1+\kappa}$ is a finite 
constant that depends only on $\kappa$. It follows from our choice of $u$ and $\eta$ that
\begin{align*}
	P(W^\pm > 0, \, (\mathcal{C}_r^\pm)^c) &= O\left(  \sum_{j = \floor{m/2}}^r  
		\frac{n^{\eta(1+\kappa)}}{u^{\kappa j}} + u^{-\kappa \floor{m/2}} 
		+ (u/\mu)^{\alpha^\pm \floor{m/2}} 1(q^\pm > 0) \right) \\
	&= O\left( n^{\eta(1+\kappa)} u^{-\kappa \floor{m/2}} 
		+ (u/\mu)^{\alpha^\pm \floor{m/2}} 1(q^\pm  > 0) \right) \\
	&= O \left( n^{-\eta} + n^{-\alpha^\pm (1/2 - \epsilon(1+\kappa)/\kappa)} 1(q^\pm > 0) \right).
\end{align*}
Hence, as $n \to \infty$, 
\begin{align*}
	&E\left[ 1(\mathcal{B}) \left| \exp\left\{ - \frac{1}{\nu n}  \sum_{i = m_n}^{k+1}  
		\tbtZout_{\ceil{i/2}} \tbtZin_{\floor{i/2}} \right\} - \exp\left\{ -\frac{\nu}{n}  
		\sum_{i = m_n}^{k+1} \mu^{i} W^+ W^-  \right\}   \right| \right]  \\
	&= O \left( n^{-\eta} + n^{-\alpha^+ (1/2 - \epsilon(1+\kappa)/\kappa)} 1(q^+ > 0)  
		+ n^{-\alpha^- (1/2 - \epsilon(1+\kappa)/\kappa)} 1(q^- > 0)  \right).
\end{align*}

Finally, on the event $\mathcal{B}^c$ we have 
\begin{align}
	&E\left[ 1(\mathcal{B}^c) \left| \exp\left\{-\frac{1}{\nu n}  \sum_{i = m_n}^{k+1}  
		\tbtZout_{\ceil{i/2}} \tbtZin_{\floor{i/2}} \right\} - \exp\left\{-\frac{\nu}{n}  
		\sum_{i = m_n}^{k+1} \mu^{i} W^- W^+  \right\}   \right| \right] \label{eq:Extinction} \\
	&=  E\left[ 1\left(\mathcal{B}^c, \, \tbtZout_{\ceil{m_n/2}} \tbtZin_{\floor{m_n/2}} > 0 \right) 
		\left| \exp\left\{-\frac{1}{\nu n} \sum_{i = m_n}^{k+1}  \tbtZout_{\ceil{i/2}} 
		\tbtZin_{\floor{i/2}} \right\} - 1 \right| \right]  \notag \\
	&\leq P\left( \mathcal{B}^c, \, \tbtZout_{\ceil{m_n/2}} \tbtZin_{\floor{m_n/2}} > 0 \right) 
		\notag  \\
	&\leq P\left( W^+ = 0, \,  \tbtZout_{\ceil{m_n/2}} > 0 \right) 
		+ P\left( W^- = 0, \, \tbtZin_{\floor{m_n/2}} > 0 \right). \notag
\end{align}
By Lemma \ref{L.Extinction}, conditionally on $W^\pm = 0$, $\{ \hat Z_k^\pm: k \geq 1\}$ is a 
delayed branching process with offspring distributions $(\tilde g^\pm, \tilde f^\pm)$, as defined 
in the lemma, having means $\nu^\pm$ and $\lambda^\pm < 1$, respectively. Moreover, by Theorem 4 in \cite{Athreya2012}, $W^\pm = 0$ implies that either $q^\pm > 0$ or $\hat Z_1^\pm = 0$. Therefore, from Markov's
inequality we obtain
\begin{align*}
	P\left( W^\pm = 0, \, \hat Z_{\floor{m_n/2}}^\pm > 0 \right) 
	&\leq E\left[ \left. \hat Z_{\floor{m_n/2}}^\pm \right| W^\pm = 0 \right]  1(q^\pm > 0) 
		= \nu^\pm (\lambda^\pm)^{\floor{m_n/2} - 1} 1(q^\pm > 0) \\
	&= O\left( n^{-(1-\epsilon)|\log \lambda^\pm|/\log \mu} 1(q^\pm > 0)  \right) 
		= O\left( n^{-(1-\epsilon) \alpha^\pm} 1(q^\pm > 0) \right).
\end{align*}

We conclude that as $n \to \infty$,
\begin{align*}
	&E\left[ \left| \exp\left\{ - \frac{1}{\nu n}  \sum_{i = m_n}^{k+1}  \tbtZout_{\ceil{i/2}} 
		\tbtZin_{\floor{i/2}} \right\} - \exp\left\{-\frac{\nu}{n}  \sum_{i = m_n}^{k+1} 
		\mu^{i} W^+ W^-  \right\} \right| \right] \\
	&= O\left( n^{-\epsilon/2} + n^{-\alpha^+ (1/2 - \epsilon(1+\kappa)/\kappa)} 1(q^+ > 0)  
		+ n^{-\alpha^- (1/2 - \epsilon(1+\kappa)/\kappa)} 1(q^- > 0)   \right).
\end{align*}
\end{proof}

The last proof in this section is that of Corollary \ref{C.FiniteHopcount}, which computes an expression for the probability that two randomly chosen nodes are connected by a directed path.

\begin{proof}[Proof of Corollary~\ref{C.FiniteHopcount}]
Fix $0 < \delta < 1/4$ and set $\omega_n = 2(1-\delta)\log n/(\log \mu)$. Our analysis is based on splitting the probability that the hopcount is finite into two terms:
$$\mathbb{P}_n(H_n < \infty) = \mathbb{P}_n(H_n \le \omega_n) + \mathbb{P}_n (\omega_n < H_n < \infty),$$
where for the first probability we will use the approximation provided by Theorem~\ref{T.MainTheoremBody}. Intuitively, the second term corresponds to an event that should be negligible in the limit, since it is unlikely that if there exists a directed path between the two randomly chosen nodes it will not have been discovered after $\omega_n$ steps of the exploration process. 

First, note that from Lemma \ref{L.Extinction} we have
\[
	s^\pm = 1 - \sum_{t=0}^\infty g^\pm(t) (q^\pm)^t \quad \text{with} \quad 
	q^\pm = \Prob{\mathcal{Z}^\pm_k = 0 \text{ for some } k \geq 1},
\] 
and since $W^+$ and $W^-$ are independent, $s^+ s^- = P(W^+ W^- > 0)$.  As in the previous proof, let $\mathcal{B} = \{ W^+ W^- > 0\}$  and write
\[
	E\left[ \exp \left\{ - \frac{\nu \mu^{k}}{(\mu-1)n} W^+ W^- \right\} \right] 
	= E\left[ \exp \left\{ - \frac{\nu \mu^{k}}{(\mu-1)n} W^+ W^- \right\} 1(\mathcal{B})\right] 
	+ P(\mathcal{B}^c).
\]
Next, use the triangle inequality followed by an application of Theorem \ref{T.MainTheoremBody} to obtain
\begin{align*}
	\left| \mathbb{P}_n(H_n < \infty) - s^+ s^- \right| 
	&= \left| \mathbb{P}_n( H_n \leq \omega_n) + \mathbb{P}_n( \omega_n < H_n < \infty) - P(\mathcal{B}) \right| \\
	&\leq \left| \mathbb{P}_n( H_n \leq \omega_n) 
	- P(\mathcal{B}) \right| + \mathbb{P}_n(\omega_n < H_n < \infty)  \\
	&\leq \left| \mathbb{P}_n(H_n \leq \omega_n) - 1 
		+ E\left[ \exp\left\{-\frac{\nu \mu^{\omega_n}}{(\mu-1)n} W^+ W^- \right\} \right] \right| \\
	&\hspace{10pt}+ \left| 1 - E\left[ \exp \left\{ - \frac{\nu \mu^{\omega_n}}{(\mu-1)n} W^+ W^- 
		\right\} \right]  - P(\mathcal{B}) \right| +  \mathbb{P}_n(\omega_n < H_n < \infty) \\
	&\leq E\left[ \exp \left\{-\frac{\nu \mu^{\omega_n}}{(\mu-1)n} W^+ W^- \right\} 1(\mathcal{B}) 
		\right] + \mathbb{P}_n(\omega_n < H_n < \infty) + O(n^{-c_1})
\end{align*}
for some $c_1 > 0$, as $n \to \infty$. 

To analyze $\mathbb{P}_n( \omega_n < H_n < \infty)$ use the expression in \eqref{eq:hopcount_main}  to see that 
for any $k \geq 0$, 
\begin{align}
	\mathbb{P}_n(k < H_n < \infty) \leq \mathbb{E}_n \left[ 1(\Zout_{\ceil{(k+1)/2}} 
		\Zin_{\floor{(k+1)/2}} > 0) \prod_{i=2}^{k+1}  \left(1 - \frac{\Zin_{\floor{i/2}}}{L_n} 
		\right)^{\Zout_{\ceil{i/2}}} \right].\label{eq:FiniteHopcountUB}
\end{align}
Note that the same steps in the proofs of Propositions \ref{P.HopCount} and \ref{P.BPapprox} give 
that \eqref{eq:FiniteHopcountUB} is equal to
\begin{align*}
	&E\left[ 1(\tbtZout_{\ceil{(k+1)/2}} \tbtZin_{\floor{(k+1)/2}} > 0) \exp\left\{-\frac{\nu 
		\mu^{k}}{(\mu-1)n} W^+ W^- \right\} \right] + O \left( n^{-a} + n^{-b} \right) \\
	&\leq E\left[ \exp\left\{ - \frac{\nu \mu^{k}}{(\mu-1)n} W^- W^+ \right\} 1(\mathcal{B}) 
		\right] + P\left( \tbtZout_{\ceil{(k+1)/2}} \tbtZin_{\floor{(k+1)/2}} > 0, \, \mathcal{B}^c\right) 
		+ O\left( n^{-a} + n^{-b} \right)
\end{align*}
for some constants $a, b > 0$ and all $0 \leq k \leq \omega_n$. It follows that
\begin{align*}
	\left| \mathbb{P}_n(H_n < \infty) - s^+ s^- \right| 
	&\leq 2 E\left[ \exp\left\{ - \frac{\nu n^{1-2\delta} }{\mu-1} W^+ W^- \right\} 
		1(\mathcal{B}) \right] \\
	&\hspace{5mm} + P\left( \tbtZout_{\ceil{(\omega_n+1)/2}} \tbtZin_{\floor{(\omega_n+1)/2}} > 0, 
		\, \mathcal{B}^c\right) + O\left( n^{-\min\{a, b, c_1\}} \right),
\end{align*}
where we have used the observation that $\mu^{\omega_n} \geq n^{2(1-\delta)}$.

Next, use Lemma \ref{L.Extinction} to see that conditionally on $W^\pm = 0$ the process 
$\{\hat Z^\pm_k: k \geq 1\}$ is a subcritical delayed branching process with offspring 
distributions $(\tilde g^\pm, \tilde f^\pm)$, defined in the lemma, having means $\nu^\pm$ and 
$\lambda^\pm < 1$, respectively. It then follows from the union bound and Markov's inequality that, 
\begin{align*}
	&P\left( \tbtZout_{\ceil{(\omega_n+1)/2}} \tbtZin_{\floor{(\omega_n+1)/2}} > 0, \, \mathcal{B}^c
		\right) \\
	&\leq P\left( \tbtZout_{\ceil{(\omega_n+1)/2}}  \geq 1, \, W^+ =0 \right) 
		+ P\left( \tbtZin_{\floor{(\omega_n+1)/2}}  \geq 1, \, W^- =0 \right) \\
	&\leq (1-s^+) \nu^+ (\lambda^+)^{\ceil{(\omega_n+1)/2} - 1} + (1-s^-) \nu^- (\lambda^-)^{
		\floor{(\omega_n+1)/2} - 1 } \\
	&= O\left( n^{-(1-\delta) \alpha^+ } 1(q^+ > 0) +  n^{-(1-\delta)\alpha^-} 1(q^- > 0) \right),
\end{align*}
where $\alpha^\pm = |\log \lambda^\pm|/\log\mu$ provided $q^\pm > 0$. 

Finally, to analyze the remaining expectation define the event $\mathcal{D} = \{ W^+ > n^{-1/4}, 
\, W^- > n^{-1/4}\}$ and note that
\begin{align*}
	E\left[ \exp\left\{-\frac{\nu n^{1-2\delta}}{\mu-1} W^+ W^- \right\} 1(\mathcal{B}) \right] 
	&\leq E\left[ \exp\left\{-\frac{\nu n^{1-2\delta}}{\mu-1} W^+ W^- \right\} 1(\mathcal{D}) 
		\right] + P(\mathcal{B} \cap \mathcal{D}^c) \\
	&\leq \exp\left\{-\frac{\nu}{(\mu-1)} \cdot n^{1/2 - 2\delta} \right\} 
		+ P(\mathcal{B} \cap \mathcal{D}^c) \\
	&\leq P(0 < W^+ \leq n^{-1/4}) + P(0 < W^- \leq n^{-1/4}) + o(n^{-1}). 
\end{align*}
The proof of Lemma \ref{L.Martingale} gives that $P(0 < W^\pm \leq n^{-1/4}) = O( n^{-1} 
1(q^\pm =0) + n^{-\alpha^\pm/4} 1(q^\pm > 0))$, which completes the proof. 
\end{proof}


\subsection{The i.i.d.~algorithm} \label{SS.TechnicalLemmata}

This last section of the appendix contains the proof of Theorem \ref{T.IIDAlgorithm}, which shows that the i.i.d.~algorithm in Section \ref{SS.IIDAlgorithm} generates bi-degree sequences that satisfty Assumption \ref{A.Wasserstein} with high probability. 

\begin{proof}[Proof of Theorem \ref{T.IIDAlgorithm}]
With some abuse of notation with respect to the proofs in the previous sections, define the events $B_n = \left\{ d_1(G_n^+, G^+) \leq n^{-\varepsilon} , \, d_1(G_n^-, G^-) \leq n^{-\varepsilon} \right\}$ and $E_n = \left\{ \sum_{i=1}^n ((D_i^-)^\kappa + (D_i^+)^\kappa) D_i^+ D_i^- \leq K_\kappa n \right\}$. Assume, without loss of generality, that $K_\kappa > E[ ((\mathscr{D}^-)^\kappa + (\mathscr{D}^+)^\kappa) \mathscr{D}^+ \mathscr{D}^-] \triangleq H_\kappa$. Next, note that
$$P(\Omega_n^c) \leq P(B_n^c) + P(E_n^c \cap B_n) + P\left(\max\{ d_1(F_n^+, F^+) , \, d_1(F_n^-, F^-) \} > n^{-\varepsilon}, \, B_n \cap E_n \right).$$

We start by showing that $P(B_n^c) \to 0$ as $n \to \infty$. To this end, let $\hat G_n^-$ and $\hat G_n^+$ denote the empirical distribution functions of $\mathscr{D}^-_1, \dots \mathscr{D}^-_n$ and $\mathscr{D}^+_1, \dots, \mathscr{D}^+_n$, respectively; note that although $\hat G_n^\pm$ is well defined regardless of the value of $\Delta_n$, $G_n^\pm, F_n^\pm$ are only defined conditionally on the event $\Psi_n = \{ |\Delta_n| \leq n^{1-\delta} \}$. Furthermore, since $E[|\mathscr{D}^- - \mathscr{D}^+|^{1/(1-\delta) }] \leq \left( E[|\mathscr{D}^- - \mathscr{D}^+|^{1+\kappa}] \right)^{(1-\delta)/(1+\kappa)} < \infty$, the Kolmogorov-Marcinkiewicz-Zygmund strong law of large numbers gives
$$P\left( \lim_{n \to \infty} \frac{\Delta_n}{n^{1-\delta}} = 0 \right) = 1,$$
and hence, $P(\Psi_n) \to 1$ as $n \to \infty$. 

It follows from the triangle inequality and Theorem 2.2 in \cite{Barr_Gin_Mat_99} (see also Proposition 3 in \cite{Chen_Olv_15}), that
\begin{align*}
E\left[  d_1(G_n^\pm, G^\pm)  \right] &\leq E\left[ \left. d_1(G_n^\pm, \hat G_n^\pm) \right| \Psi_n  \right] + E\left[ \left. d_1(\hat G_n^\pm, G^\pm) \right| \Psi_n \right] \\
&\leq \frac{1}{P(\Psi_n)} \left( E\left[ d_1(G_n^\pm, \hat G_n^\pm) 1(\Psi_n) \right] + K_\delta \, n^{-\delta} \right)
\end{align*}
for some finite constant $K_\delta$. Moreover, on the event $\Psi_n$, 
\begin{align*}
d_1(G_n^+, \hat G_n^+) &= \frac{1}{n} \int_0^\infty \left| \sum_{i=1}^n  \left( 1(\mathscr{D}^-_i+ \tau_i \leq x) - 1(\mathscr{D}^-_i \leq x) \right) \right| dx = \frac{1}{n} \int_0^\infty \sum_{i=1}^n 1(\mathscr{D}^-_i \leq x < \mathscr{D}^-_i + \tau_i) \, dx \\
&= \frac{1}{n} \sum_{i=1}^\infty \tau_i \leq \frac{|\Delta_n|}{n} \leq n^{-\delta}. 
\end{align*}
Since the analysis of $d_1(G_n^-, \hat G_n^-)$ is the same, we obtain
$$E\left[ \left. d_1(G_n^\pm, G^\pm) \right| \Psi_n \right]  \leq \frac{1}{P(\Psi_n)} \left( n^{-\delta} + K_\delta \, n^{-\delta} \right),$$
from which it follows that as $n \to \infty$,
$$P( B_n^c ) \leq n^\varepsilon \left( E\left[ d_1(G_n^+, G^+)  \right] + E\left[ d_1(G_n^-, G^-) \right] \right) = O\left( n^{-\delta + \varepsilon} \right).$$

Next, to analyze $P(E_n^c \cap B_n)$, note that $\tau_i \chi_i = 0$, and therefore
\begin{align*}
&\sum_{i=1}^n ((D_i^-)^\kappa + (D_i^+)^\kappa) D_i^+ D_i^- \\
&\leq \sum_{i=1}^n ( (\mathscr{D}^-_i+\tau_i)^\kappa + (\mathscr{D}^+_i+\chi_i)^\kappa) (\mathscr{D}^+_i + \chi_i) (\mathscr{D}^-_i+ \tau_i) \\
&\leq  \sum_{i=1}^n ( (\mathscr{D}^-_i)^\kappa +\tau_i  + (\mathscr{D}^+_i)^\kappa +\chi_i) (\mathscr{D}^+_i \mathscr{D}^-_i + \mathscr{D}^+_i \tau_i + \mathscr{D}^-_i \chi_i) \\
&= \sum_{i=1}^n \left\{ ( (\mathscr{D}^-_i)^\kappa  + (\mathscr{D}^+_i)^\kappa) \mathscr{D}^+_i \mathscr{D}^-_i +  \tau_i  (\mathscr{D}^+_i \mathscr{D}^-_i + \mathscr{D}^+_i  ) + \chi_i (\mathscr{D}^+_i \mathscr{D}^-_i + \mathscr{D}^-_i ) \right\} .
\end{align*}
Now set $Y_i = ( (\mathscr{D}^-)_i^\kappa  + (\mathscr{D}^+_i)^\kappa) \mathscr{D}^+_i \mathscr{D}^-_i - H_\kappa$ and $W_i = \tau_i  (\mathscr{D}^+_i \mathscr{D}^-_i + \mathscr{D}^+_i  ) + \chi_i (\mathscr{D}^+_i \mathscr{D}^-_i + \mathscr{D}^-_i )$  to obtain
\begin{align*}
P(E_n^c \cap B_n) &\leq P\left( \left. \sum_{i=1}^n Y_i + \sum_{i=1}^n W_i > (K_\kappa - H_\kappa) n, \, B_n \right| \Psi_n \right) \\
&\leq \frac{1}{P(\Psi_n)} P\left(  \sum_{i=1}^n Y_i > n(K_\kappa-H_\kappa)/2  \right)  + P\left(  \left. \sum_{i=1}^n W_i  > n(K_\kappa - H_\kappa)/2, \, B_n \right| \Psi_n \right).
\end{align*}
Since $n^{-1} \sum_{i=1}^n Y_i \to 0$ almost surely by the strong law of large numbers, the first probability converges to zero as $n \to \infty$. For the second probability use Markov's inequality to obtain
\begin{align*}
P\left(  \left. \sum_{i=1}^n W_i  > n(K_\kappa - H_\kappa)/2, \, B_n \right| \Psi_n \right) &\leq \frac{2E[ W_1 1(\Psi_n \cap B_n)]}{P(\Psi_n) (K_\kappa-H_\kappa)} .
\end{align*}
Now note that
\begin{align*}
E[W_1 1(\Psi_n \cap B_n)] &= E\left[  E[\tau_1 | \{(\mathscr{D}^-_i, \mathscr{D}^+_i)\}_{i=1}^n ] (\mathscr{D}^-_1\mathscr{D}^+_1 + \mathscr{D}^+_1) 1(\Psi_n \cap B_n) \right]  \\
&\hspace{5mm} +  E\left[ E[\chi_1 | \{(\mathscr{D}^-_i, \mathscr{D}^+_i)\}_{i=1}^n ] (\mathscr{D}^-_1 \mathscr{D}^+_1 + \mathscr{D}^-_1)  1(\Psi_n \cap B_n) \right] \\
&= E\left[ \frac{\Delta_n^-}{L_n}  (\mathscr{D}^-_1 \mathscr{D}^+_1 + \mathscr{D}^+_1) 1(\Psi_n \cap B_n) \right] + E\left[ \frac{\Delta_n^+}{L_n}  (\mathscr{D}^-_1 \mathscr{D}^+_1 + \mathscr{D}^-_1) 1(\Psi_n \cap B_n) \right] \\
&\leq n^{1-\delta} E\left[ \frac{(\mathscr{D}^-_1 \mathscr{D}^+_1 + \mathscr{D}^+_1 + \mathscr{D}^-_1)}{L_n} 1(B_n) \right] \leq \frac{n^{-\delta}}{\nu - n^{-\varepsilon}} E\left[ \mathscr{D}^-_1 \mathscr{D}^+_1 + \mathscr{D}^+_1 + \mathscr{D}^-_1 \right],
\end{align*}
where in the last step we used the observation that on the event $B_n$ we have $L_n = n \nu_n \geq n (\nu - n^{-\varepsilon})$, since $|\nu_n - \nu| \leq d_1(G_n^+, G^+)$. Hence, we have shown that $P(E_n^c \cap B_n) \to 0$ as $n \to \infty$. 

For the size-biased distributions note that
\begin{align*}
d_1(F_n^+, F^+) &\leq \int_0^\infty \left| \frac{1}{L_n} - \frac{1}{\nu n} \right|  \sum_{i=1}^n 1(D_i^- > x) D_i^+ \, dx + \frac{1}{n\nu }  \int_0^\infty \left|   \sum_{i=1}^n \left( 1(D_i^- > x) D_i^+ - 1(\mathscr{D}^-_i > x) \mathscr{D}^+_i \right)  \right|  dx \\
&\hspace{5mm} +  \int_0^\infty \left|  \frac{1}{\nu n} \sum_{i=1}^n 1(\mathscr{D}^-_i > x) \mathscr{D}^+_i - 1 + F^+(x) \right| dx \\
&=\left| \frac{\nu  - \nu_n}{\nu \nu_n} \right| \frac{1}{n}  D_i^- D_i^+ + \frac{1}{\nu n} \int_0^\infty \sum_{i=1}^n \left( 1(\mathscr{D}^-_i \leq x < \mathscr{D}^-_i + \tau_i) \mathscr{D}^+_i + 1(\mathscr{D}^-_i > x) \chi_i \right) dx \\
&\hspace{5mm} + \int_0^\infty \left| \frac{1}{n} \sum_{i=1}^n X_i^+(x) \right| dx,
\end{align*}
where $X_i^+(x) = 1(\mathscr{D}^-_i > x)\mathscr{D}^+_i/\nu - 1+F^+(x)$. Now recall that $|\nu_n - \nu| \leq d_1(G_n^+, G^+)$, 
and therefore, on the event $B_n \cap E_n$,  
\begin{align*}
 d_1(F_n^+, F^+) &\leq \left| \frac{\nu  - \nu_n}{\nu (\nu - n^{-\varepsilon}) } \right| K_\kappa +   \frac{L_n}{n^2 \nu (\nu - n^{-\varepsilon})}  \sum_{i=1}^n (\tau_i \mathscr{D}^+_i + \chi_i \mathscr{D}^-_i)  + \int_0^\infty \left| \frac{1}{n} \sum_{i=1}^n X_i^+(x) \right| dx.
\end{align*}

Since the case $d_1(F_n^-, F^-)$ is symmetric by setting $X_i^-(x) = 1(\mathscr{D}^+_i > x) \mathscr{D}^-_i/\nu - 1 + F^-(x)$,  it follows that
\begin{align*}
&P\left( d_1(F_n^\pm, F^\pm) > n^{-\varepsilon}, \, B_n \cap E_n \right) \\
&\leq \frac{1}{P(\Psi_n)} P\left( \left| \frac{\nu  - \nu_n}{\nu (\nu - n^{-\varepsilon}) } \right| K_\kappa +   \frac{1(\Psi_n) L_n}{n^2 \nu (\nu - n^{-\varepsilon})}  \sum_{i=1}^n (\tau_i \mathscr{D}^+_i + \chi_i \mathscr{D}^-_i)  + \int_0^\infty \left| \frac{1}{n} \sum_{i=1}^n X_i^\pm (x) \right| dx > n^{-\varepsilon} \right) \\
&\leq \frac{n^\varepsilon}{P(\Psi_n)} \left( \frac{K_\kappa E\left[ d_1(G_n^+, G^+) \right] }{\nu (\nu - n^{-\varepsilon})} + \frac{ E\left[  1(\Psi_n) L_n (\tau_1 \mathscr{D}^+_1 + \chi_1 \mathscr{D}^-_1) \right] }{n\nu(\nu - n^{-\varepsilon})} + \frac{1}{n} \int_0^\infty E\left[ \left| \sum_{i=1}^n X_i^\pm (x) \right| \right] dx  \right)  \\
&= O\left( n^{-\delta+\varepsilon} + \frac{E\left[ 1(\Psi_n)  L_n (\tau_1 \mathscr{D}^+_1 + \chi_1 \mathscr{D}^-_1) \right]}{n^{1-\varepsilon}} + \frac{1}{n^{1-\varepsilon}} \int_0^\infty E\left[ \left| \sum_{i=1}^n X_i^\pm (x) \right| \right] dx  \right).
\end{align*}
To bound the middle term in the last expression, note that
\begin{align*}
E\left[ 1(\Psi_n)  L_n (\tau_1 \mathscr{D}^+_1 + \chi_1 \mathscr{D}^-_1) \right] &= E\left[ 1(\Psi_n)  L_n ( E[\tau_1 |\{ \mathscr{D}^-_i, \mathscr{D}^+_i)\}_{i=1}^n] \mathscr{D}^+_1 + E[\chi_1 |\{ \mathscr{D}^-_i, \mathscr{D}^+_i)\}_{i=1}^n]  \mathscr{D}^-_1) \right] \\
&= E\left[ 1(\Psi_n)   ( |\Delta_n| \mathscr{D}^+_1 + |\Delta_n|  \mathscr{D}^-_1) \right] \leq n^{1-\delta} E[ \mathscr{D}^+ + \mathscr{D}^-].
\end{align*}

To complete the proof, choose $1/(1-\varepsilon) < p < 1+\kappa$, use the monotonicity of the norm $\| X \|_p = \left( E[ |X|^p ] \right)^{1/p}$ and apply Lemma \ref{L.Burkholder} to obtain
\begin{align*}
\frac{1}{n^{1-\varepsilon}} \int_0^\infty E\left[ \left| \sum_{i=1}^n X_i^\pm (x) \right| \right] dx &\leq \frac{1}{n^{1-\varepsilon}} \int_0^\infty  \left\| \sum_{i=1}^n X_i^\pm (x) \right\|_{p} dx \\
&\leq \frac{1}{n^{1-\varepsilon}}   \int_0^{\infty} \left( Q_p n E\left[ \left| X_1^\pm (x) \right|^p  \right] \right)^{1/p} dx \\
&= \frac{(Q_p)^{1/p}}{\nu} n^{1/p -1 + \varepsilon} \int_0^\infty \left\| \nu X_1^\pm(x) \right\|_p dx ,
\end{align*}
for some finite constant $Q_p$ that depends only on $p$; note that our choice of $p$ implies that $n^{1/p - 1+ \varepsilon} \to 0$. It only remains to verify that the integral is finite. To do this first use Minkowski's inequality to get
\begin{align*}
\left\| \nu X_1^+(x) \right\|_{p} &= \left\| 1(\mathscr{D}^- > x) \mathscr{D}^+ - E[ 1(\mathscr{D}^- > x) \mathscr{D}^+] \right\|_{p} \leq \left\| 1(\mathscr{D}^- > x) \mathscr{D}^+ \right\|_{p} + E[ 1(\mathscr{D}^- > x) \mathscr{D}^+] .
\end{align*}
Furthermore, for $x \geq 1$, 
\begin{align*}
	\left\| 1(\mathscr{D}^- > x) \mathscr{D}^+ \right\|_{p} &= \left( E\left[ 1(\mathscr{D}^- > x) (\mathscr{D}^+)^p \right] \right)^{1/p} \leq \left( E\left[ (\mathscr{D}^-/x)^{1+\kappa}  (\mathscr{D}^+)^p \right] \right)^{1/p} \\
	&= \left( E\left[ (\mathscr{D}^-)^{1+\kappa}  (\mathscr{D}^+)^p \right] \right)^{1/p} x^{-(1+\kappa)/p},
\end{align*}
while for $0 < x < 1$, $\left\| 1(\mathscr{D}^- > x) \mathscr{D}^+ \right\|_{p} = \left\| 1(\mathscr{D}^- \geq 1) \mathscr{D}^+ \right\|_{p} \leq \left( E[ \mathscr{D}^- (\mathscr{D}^+)^p ] \right)^{1/p}$. It follows that
$$\int_0^\infty \left\| \nu X_1^+(x) \right\|_p dx \leq \left( E[ \mathscr{D}^- (\mathscr{D}^+)^p ] \right)^{1/p} + \int_1^\infty \left( E\left[ (\mathscr{D}^-)^{1+\kappa}  (\mathscr{D}^+)^p \right] \right)^{1/p} x^{-(1+\kappa)/p} dx + E[ \mathscr{D}^- \mathscr{D}^+ ] < \infty.$$
The proof for $X_1^-(x)$ is obtained by exchanging the roles of $\mathscr{D}^-$ and $\mathscr{D}^+$.  
\end{proof}

\bibliographystyle{plain}
\bibliography{directeddistances}

\end{document}